\documentclass[11pt,reqno]{amsart}
\usepackage{xifthen}
\usepackage{pkgfile}
\newtheorem{example}{Example}

\title[Universal Limit Theorems in Graph Coloring Problems]{Universal Limit Theorems in Graph Coloring Problems With Connections to Extremal Combinatorics}

\author{Bhaswar B. Bhattacharya}
\address{Department of Statistics, Stanford University, California, USA,
{\tt bhaswar@stanford.edu}}

\author{Persi Diaconis}
\address{Department of Mathematics and Statistics, Stanford University, California, USA,
{\tt diaconis@math.stanford.edu}}

\author{Sumit Mukherjee}
\address{Department of Statistics, Stanford University, California, USA, {\tt sumitm@stanford.edu}}

\begin{document}

\begin{abstract}This paper proves limit theorems for the number of monochromatic edges in uniform random colorings of general random graphs. These can be seen as generalizations of the birthday problem (what is the chance that there are two friends with the same birthday?). It is shown that if the number of colors grows to infinity, the asymptotic distribution is either a Poisson mixture or a Normal depending solely on the limiting behavior of the ratio of the number of edges in the graph and the number of colors. This result holds for any graph sequence, deterministic or random. On the other hand, when the number of colors is fixed, a necessary and sufficient condition for asymptotic normality is determined.  Finally, using some results from the emerging theory of dense graph limits, the asymptotic (non-normal) distribution is characterized for any converging sequence of dense graphs. The proofs are based on moment calculations which relate to the results of Erd\H os and Alon on extremal subgraph counts. As a consequence, a simpler proof of a result of Alon, estimating the number of isomorphic copies of a cycle of given length in graphs with a fixed number of edges, is presented.
\end{abstract}


\maketitle

\section{Introduction}

Suppose the vertices of a finite graph $G$ are colored independently and uniformly at random with $c$ colors. The probability that the resulting coloring has no monochromatic edge, that is, it is a proper coloring, is $\chi_G(c)/c^{|V(G)|}$, where $\chi_G(c)$ denotes the number of proper colorings of $G$ using $c$-colors and $|V(G)|$ is the number of vertices in $G$. The function $\chi_G$ is the chromatic polynomial of $G$, and is a central object in graph theory \cite{chromaticbook,toft_book,toft_unsolved}. This paper studies the limiting distribution of the number of monochromatic edges in uniform random colorings of general random graphs. 

\subsection{Universal Limit Theorems For Monochromatic Edges} Let $\sG_n$ denote the space of all simple undirected graphs on $n$ vertices labeled by $[n]:=\{1,2,\cdots, n\}$. Given a graph $G_n\in\sG_n$ with adjacency matrix $A(G_n)=((A_{ij}(G_n)))_{1\leq i, j \leq n}$, denote by $V(G_n) $ the set of vertices, and by $E(G_n)$ the set of edges of $G_n$, respectively.
The vertices of $G_n$ are colored with $c=c(n)$ colors as follows:  
\begin{equation}\P(v\in V(G_n) \text{ is colored with color } a\in \{1, 2, \ldots, c\}|G_n)=\frac{1}{c},
\label{eq:uniform}
\end{equation}
independent from the other vertices.
If $Y_i$ is the color of vertex $i$, then
\begin{equation}
N(G_n):=\sum_{1\le i<j\le n}A_{ij}(G_n)\pmb 1\{Y_i=Y_j\}=\sum_{(i, j)\in E(G_n)}\pmb 1\{Y_i=Y_j\},
\label{eq:medges}
\end{equation}
denotes the number of monochromatic edges in the graph $G_n$. Note that $\mathbb{P}(N(G_n)=0)$ is the probability that $G_n$ is properly colored. When $c=365$ and $G_n=K_n$ is a complete graph this reduces to the classical birthday problem: $\mathbb{P}(N(K_n)\geq 1)$ is the probability that there are two people with the same birthday in a room with $n$ people. We study the limiting behavior of $N(G_n)$ as the size of the graph becomes large, allowing the graph itself to be random, under the assumption that the joint distribution of $(A(G_n),\vec {Y_n})$ is mutually independent, where $\vec {Y_n}=(Y_1, Y_2, \ldots, Y_n)$ are i.i.d. random variables with  $\mathbb{P}(Y_1=a)=1/c$, for all $a\in [c]$. Note that this setup includes the case where $\{G_1, G_2, \ldots \}$ is a deterministic (non-random) graph sequence, as well.

An application of the easily available version of Stein's method gives a general limit theorem for $N(G_n)$ that works for all color distributions \cite{barbourholstjanson,birthdayexchangeability}. For the uniform coloring scheme, using the fact that the random variables are pairwise independent, Barbour et al. \cite{barbourholstjanson} proved a Poisson approximation for the number of monochromatic edges which works for any sequence of deterministic graphs.  The following theorem gives a new proof and slightly extends this result by showing that the same is true for random graphs. Unlike Stein's method, our proof, which is based on the method of moments, does not give convergence rates. However, it illustrates the connections to extremal combinatorics, and builds up to our later results.


\begin{thm}\label{th:birthdayuniformuniversal}
Let $G_n\in \sG_n$ be a random graph sampled according to some probability distribution over $\sG_n$ and $c = c(n) \to \infty$. Then under the uniform coloring distribution, the following is true:
$$
N(G_n)\stackrel{\sD}\rightarrow \left\{
\begin{array}{ccc}
0  & \text{if}  &  \frac{1}{c}\cdot |E(G_n)|\stackrel{\sP}{\rightarrow} 0,\\
\infty  & \text{if}  &  \frac{1}{c}\cdot |E(G_n)|\stackrel{\sP}{\rightarrow} \infty,\\
W  & \text{if}  &  \frac{1}{c}\cdot |E(G_n)|\stackrel{\sD}{\rightarrow} Z;
\end{array}
\right.
$$
where $\mathbb{P}(W=k)=\frac{1}{k!}\mathbb{E}(e^{-Z}Z^k)$. In other words, $W$ is distributed as a mixture of Poisson random variables mixed over the random variable $Z$.
\label{th:poissonuniversal}
\end{thm}


Theorem \ref{th:poissonuniversal} is universal because it only depends on
the limiting behavior of $|E(G_n)|/c$ and it works for any graph sequence $\{G_n\}_{n\geq 1}$, deterministic or random. The theorem is proved using the method of moments, that is, the conditional moments of $N(G_n)$ are compared with conditional moments of the random variable
 \begin{equation}\label{eq:independentbernoulli}
 M(G_n):=\sum_{1\le i<j\le n}A_{ij}(G_n)Z_{ij},
 \end{equation}
where $\{Z_{ij}\}_{(i, j)\in E(G_n)}$ are independent $\dBer(1/c)$. The combinatorial quantity that needs to be bounded during moment calculations is the number of isomorphic copies of a graph $H$ in another graph $G$, to be denoted by $N(G, H)$. Using spectral properties of the adjacency matrix of $G$, we estimate $N(G, H)$, when $H=C_g$ is a $g$-cycle, This result is then used to show the asymptotic closeness of the conditional moments of $N(G_n)$ and $M(G_n)$.

Theorem \ref{th:birthdayuniformuniversal} asserts that if $\frac{1}{c}|E(G_n)|\stackrel{\sP}{\rightarrow} \infty$, then $N(G_n)$ goes to infinity in probability. Since a Poisson random variable with mean growing to infinity converges to a standard normal distribution after centering by the mean and scaling by the standard deviation, it is natural to wonder whether the same is true for $N(G_n)$. This is not true in general if $|E(G_n)|/c$ goes to infinity, with $c$ fixed as indicated in Example \ref{ex:K2n} and Example \ref{ex:completegraph}. On the other hand, if $c\rightarrow \infty$, an off-the-shelf version of Stein's method can be used to show the normality of $N(G_n)$ under some extra condition on the structure of the graph. However, under the uniform coloring scheme, using the method of moments argument and extremal combinatorics estimates from Alon \cite{alon81}, it can be shown that the normality for $N(G_n)$ 
 is universal whenever $c\rightarrow \infty$.

\begin{thm}Let $G_n\in \sG_n$ be a random graph sampled according to some probability distribution over $\sG_n$. Then for any uniform $c$-coloring of
$G_n$, with $c = c(n) \to \infty$ and $|E(G_n)|/c\stackrel{\sP}\rightarrow \infty$,
$$\frac{1}{\sqrt{|E(G_n)|/c}}\left(N(G_n)-\frac{|E(G_n)|}{c}\right)\stackrel{\sD}{\rightarrow}N(0,1).$$
\label{th:normal}
\end{thm}


In the proof of Theorem \ref{th:normal} the conditional central moments of $N(G_n)$ are compared with the conditional central moments of $M(G_n)$. In this case, a combinatorial quantity involving the number of multi-subgraphs of $G_n$ shows up. Bounding this quantity requires extensions of Alon's \cite{alon81,alon86} results to multi-graphs and leads to some other results in graph theory which may be of independent interest. Error rates for the above CLT was proved recently by Fang \cite{xiao}.

\subsection{Normality For Fixed Number of Colors}

The limiting distribution of $N(G_n)$ might not be asymptotically normal if $|E(G_n)|\rightarrow \infty$, but the number of colors $c$ is fixed. In fact, in Example \ref{ex:K2n} it will be the shown that if the graph $K_{2, n}$ is uniformly colored with $c=2$ colors, then the limiting distribution of $N(G_n)$ is a mixture of a standard normal and point mass at 0.

However, for many graph sequences the limiting distribution is asymptotically normal. To characterize graph sequences for which asymptotic normality holds we introduce the following definition:


\begin{defn}
A deterministic sequence of graphs $\{G_n\}_{n\geq 1}$ is said to satisfy the {\it asymptotic 4-cycle free (ACF4) condition} if 
\begin{equation}
N(C_4, G_n)=o(|E(G_n)|^{2}).
\label{eq:acf}
\end{equation}
A sequence of random graphs $\{G_n\}_{n\geq 1}$ is said to satisfy the {\it ACF4 condition in probability} if 
\begin{equation}
N(C_4, G_n)=o_P(|E(G_n)|^{2}).
\label{eq:acf}
\end{equation}
(The notation $X_n=o_P(a_n)$ means that $X_n/a_n$ converges to zero in probability as $n\rightarrow \infty$.)
\end{defn}


The following theorem shows that the ACF4 condition is necessary and sufficient for the normality of $N(G_n)$ when the number of colors $c$ is fixed.  The proof proceeds along similar lines as in Theorem \ref{th:normal}. However, in this case, more careful estimates are required to bound the number of multi-subgraphs of $G_n$.

\begin{thm}Let $G_n\in \sG_n$ be a random graph sampled according to some probability distribution over $\sG_n$. Then for any uniform $c$-coloring of $G_n$, with $c\geq 2$, fixed and $|E(G_n)|\rightarrow\infty$,
\begin{equation}\frac{1}{\sqrt{|E(G_n)|/c}}\left(N(G_n)-\frac{|E(G_n)|}{c}\right)\stackrel{\sD}{\rightarrow}N\left(0,1-\frac{1}{c}\right),
\label{eq:normal_fixed}
\end{equation}
if and only if $\{G_n\}_{n\geq 1}$ satisfies the ACF4 condition in probability.
\label{th:normal_fixed}
\end{thm}

For the case $c=2$, the random variable $N(G_n)-|E(G_n)|/c$ can be rewritten as a quadratic form as follows: Let $\vec X=(X_1, X_2, \cdots, X_{|V(G_n)|})'$ be a vector of independent Rademacher random variables, then
$$N(G_n)-\frac{|E(G_n)|}{c}=\frac{1}{2}\sum_{i\leq j}A_{ij}(G_n)X_{i}X_j.$$
The asymptotic normality of $V_n:=\sum_{i\leq j} a_{ij}X_iX_j$ for a general sequence of symmetric matrices $\{\pmb A_n=((a_{ij}))\}$ and i.i.d. real-valued random variables $X_1, X_2, \cdots, X_n$ with zero mean, unit variance and finite fourth moment, is a well-studied problem.  The classical sufficient condition for asymptotic normality is (refer to Rotar \cite{rotar}, Hall \cite{hall}, de Jong \cite{dejong} for further details) 
\begin{equation}
\lim_{n\rightarrow \infty} \sigma_n^{-4}\E(V_n-\E(V_n))^4=3, \text{ and } \lim_{n\rightarrow \infty}\sigma_n^{-2}\max_{i \in [n]}\sum_{j=1}^n a^2_{ij}=0,
\label{eq:quadratic_form_condition}
\end{equation}
where $\sigma_n^2=\frac{1}{2}\tr(\pmb A_n^2)=\Var(V_n)$.
This condition is also necessary when  $X_1, X_2, \cdots, X_n$ are i.i.d. $N(0, \tau^2)$, for some $\tau>0$. 
Error bounds were obtained by G\"otze and Tikhomirov \cite{gotze,gotze_applications}. Recently, Nourdin et al. \cite{greater_three} showed that the first condition in (\ref{eq:quadratic_form_condition}) is sufficient for asymptotic normality whenever $\E(X_1^4)\geq 3$. This is an example of the {\it Fourth Moment Phenomenon} which asserts that for many sequences of non-linear functionals of random fields, a CLT is simply implied by the convergence of the corresponding sequence of fourth moments (refer to Nourdin et al. \cite{weiner_chaos} and the references therein for further details).

For the case where $X_1, X_2, \cdots, X_n$ are Rademacher variables, error bounds were also proved by Chatterjee \cite{chatterjee_normal_approximation}. Later, Nourdin et al. \cite{rademacher} showed that in this case, the first condition in (\ref{eq:quadratic_form_condition}) is necessary and sufficient for asymptotic normality. For the special case where the matrix $\pmb A_n$ is the adjacency matrix of a graph, it is easy to see that the fourth moment condition in (\ref{eq:quadratic_form_condition}) is equivalent to the ACF4 condition, making Theorem \ref{th:normal_fixed} an instance of the Fourth Moment Phenomenon. For the case of graphs, Theorem \ref{th:normal_fixed} reconstructs the result about quadratic forms for $c=2$, and extends it  for general $c\geq 3$ when $N(G_n)$ can no longer be written as a single quadratic form. 

\subsection{Limiting Distribution For Converging Sequence of Dense Graphs}

As discussed above, asymptotic normality of the number of monochromatic edges, for a fixed number of colors, does not hold when the ACF4 condition is not satisfied. 
In particular, any sequence of dense graphs $G_n$ with $\Theta(n^2)$ edges does not satisfy the ACF4 condition and hence the limiting distribution of $N(G_n)$ is not asymptotically normal. This raises the question of characterizing the limiting distribution of $N(G_n)$ for dense graphs. Recently, Lov\'asz and coauthors \cite{graph_limits_I,graph_limits_II,lovasz_book} developed a limit theory for dense graphs. Using results form this theory, we obtain the limiting distribution of $N(G_n)$ for any converging sequence of dense graphs.

\subsubsection{Graph Limit Theory}Graph limit theory connects various topics such as graph homomorphisms, Szemer\'edi regularity lemma, quasirandom graphs, graph testing and extremal graph theory, and has even found applications in statistics and related areas \cite{chatterjee_pd}. For a detailed exposition of the theory of graph limits refer to Lov\'asz \cite{lovasz_book}. In the following, we mention the basic definitions about convergence of graph sequences. If $F$ and $G$ are two graphs, then
$$t(F,G) :=\frac{|\hom(F,G)|}{|V (G)| |V (F)|},$$
where  $|\hom(F,G)|$ denotes the number of homomorphisms of $F$ into $G$. In fact, $t(F, G)$ denotes the probability that a random mapping $\phi: V (F) \rightarrow V (G)$ defines a graph homomorphism. The basic definition is that a sequence $G_n$ of graphs converges if $t(F, G_n)$ converges for every graph $F$.

There is a natural limit object in the form of a function $W \in \sW$, where $\sW$ is the space of all measurable functions from $[0, 1]^2$ into $[0, 1]$ that satisfy $W(x, y) = W(y,x)$ for all $x, y$. Conversely, every such function arises as the limit of an appropriate graph sequence. This limit object determines all the limits of subgraph densities:
if $H$ is a simple graph with $V (H)= \{1, 2, \ldots, |V(H)|\}$, let
$$t(H,W) =\int_{[0,1]^{|V(H)|}}\prod_{(i,j)\in E(H)}W(x_i,x_j) dx_1dx_2\cdots dx_{|V(H)|}.$$
A sequence of graphs $\{G_n\}_{n\geq 1}$ is said to converge to $W$ if for every finite simple graph $H$,
\begin{equation}
\lim_{n\rightarrow \infty}t(H, G_n) = t(H, W).
\label{eq:graph_limit}
\end{equation}
A sequence of random graphs $\{G_n\}_{n\geq 1}$ is said to converge to $W$ in distribution, if the sequence $$\{t(H, G_n): H \text{ is a finite simple graph}\}\dto\{t(H, W): H \text{ is a finite simple graph}\},$$ in $[0, 1]^\N$. The limit objects, that is, the elements of $\sW$, are called {\it graph limits} or {\it graphons}. A finite simple graph $G$ on $[n]$ can also be represented as a graph limit in a natural way: Define $f^G(x, y) =\boldsymbol 1\{(\ceil{nx}, \ceil{ny})\in E(G)\}$, that is, partition $[0, 1]^2$ into $n^2$ squares of side length $1/n$, and define $f^G(x, y)=1$ in the $(i, j)$-th square if $(i, j)\in E(G)$ and 0 otherwise. Observe that $t(H, f^{G}) = t(H,G)$ for every simple
graph $H$ and therefore the constant sequence $G$ converges to the graph limit $f^G$.  It turns out that the notion of convergence in terms of subgraph densities outlined above can be suitably metrized using the so-called {\it cut distance} (refer to \cite[Chapter 8]{lovasz_book}).

%

Every function $W\in \sW$ defines an operator $T_W: L_2[0, 1]\rightarrow L_2[0, 1]$, by $$(T_Wf)(x)=\int_0^1W(x, y)f(y)dy.$$
$T_W$ is a Hilbert-Schmidt operator, which is compact and has a discrete spectrum, that is, a countable multiset of non-zero real eigenvalues $\{\lambda_i(W)\}_{i \in \N}$. In particular, every non-zero eigenvalue has finite multiplicity and
$$\sum_{i=1}^\infty\lambda_i^2(W)=\int_{[0, 1]^2}W(x, y)^2dxdy:=||W||_2^2.$$

\subsubsection{Limiting Distribution of $N(G_n)$ for Dense Graphs}
Under the assumption that the sequence of random graphs $G_n$ converges to a limit $W$ in distribution, one can derive the limiting distribution of $N(G_n)$ in terms of the eigenvalues of $T_W$, whenever $\int_{[0, 1]^2}W(x, y)dxdy>0$ almost surely. Since $|E(G_n)|/n^2\pto\frac{1}{2}\int_{[0, 1]^2}W(x, y)dxdy$, the positivity condition just ensures that the graph sequence $G_n$ has $\Theta(n^2)$ edges with high probability, that is, it is dense.

\begin{thm}\label{th:chisquare}
Let $\{G_n\}_{n\geq 1}$ be a sequence of random graphs converging in distribution to a graphon $W\in\sW$, such that
almost surely $\int_{[0, 1]^2}W(x, y)dxdy>0$. Then for any uniform $c$-coloring of
$G_n$, with $c\geq 2$ fixed and $|E(G_n)|\rightarrow\infty$,
\begin{equation}
\frac{1}{\sqrt{2|E(G_n)|}}\left(N(G_n)-\frac{|E(G_n)|}{c}\right)\stackrel{\sD}{\rightarrow}\frac{1}{2c}\sum_{i=1}^\infty\left(\frac{\lambda_i(W)}{(\sum_{j=1}^\infty\lambda^2_j(W))^{\frac{1}{2}}}\right)\xi_i,
\label{eq:chisquare}
\end{equation}
where $\{\xi_i\}_{i\in \N}$ are i.i.d. $\chi^2_{(c-1)}-(c-1)$ random variables independent of $W$.
\label{th:chisquare}
\end{thm}

This theorem gives a characterization of the limiting distribution of the number of monochromatic edges for all converging sequence of dense graphs. As before, the main idea of the proof of Theorem \ref{th:chisquare} is moment comparison. However, in this case, the conditional central moments of $N(G_n)$ are compared with the conditional moments  of a random variable obtained by replacing the color at every vertex with an independent and appropriately chosen $c$-dimensional normal, which is then shown to converge in distribution to a weighted sum of independent centered chi-square $\chi^2_{(c-1)}$ random variables. Refer to Beran \cite{beran} for more about such distributions.

\subsection{Connections to Extremal Combinatorics} The combinatorial quantity that shows up in moment computations for the above theorems is $N(G, H)$, the number of isomorphic copies of a graph $H$ in another graph $G$. The quantity $N(\ell, H):=\sup_{G: |E(G)|=\ell}N(G, H)$ is a well-known object in extremal graph theory that was first studied by Erd\H os \cite{erdoscompletesubgraphs} and later by Alon \cite{alon81,alon86}. Alon \cite{alon81} showed that for any simple graph $H$ there exists a graph parameter $\gamma(H)$ such that $N(\ell, H)=\Theta(\ell^{\gamma(H)})$. Friedgut and Kahn \cite{hypergraphcopies} extended this result to hypergraphs and identified the exponent $\gamma(H)$ as the fractional stable number of the hypergraph $H$. Alon's result can be used to obtain a slightly more direct proof of Theorem \ref{th:poissonuniversal}. However, our estimates of $N(G, C_g)$ using the spectral properties of $G$ lead to a new and elementary proof of the following result of Alon \cite{alon81}:

\begin{thm}[Theorem B, Alon \cite{alon81}]If $H$ has a spanning subgraph which is a disjoint union of cycles and isolated edges, then
$$N(\ell, H)=(1+O(\ell^{-1/2}))\cdot \frac{1}{|Aut(H)|}\cdot (2\ell)^{|V(H)|/2},$$
where $|Aut(H)|$ denotes the number of automorphisms of $H$.
\label{th:alon_cycle}
\end{thm}

The above theorem calculates the exact asymptotic behavior of $N(\ell, H)$ for graphs $H$ that have a spanning subgraph consisting of a disjoint union of cycles and isolated edges. There are only a handful of graphs for which such exact asymptotics are known \cite{alon81,alon86}. Alon's proof in \cite{alon81} uses a series of combinatorial lemmas. We hope the short new proof presented in this paper is of independent interest.

The quantity $\gamma(H)$ is a well studied object in graph theory and discrete optimization and is related to the fractional stable set polytope \cite{schrijver}. While proving Theorems \ref{th:poissonuniversal} and \ref{th:normal}, we discover several interesting facts about the exponent $\gamma(H)$, which might be useful in graph theory as well. Alon \cite{alon86} showed that $\gamma(H)\leq |E(H)|$, and the equality holds if and only if $H$ is a disjoint union of stars. This is improved to $\gamma(H)\leq |V(H)|-\nu(H)$, where  $\nu(H)$ is the number of connected components of $H$ and the condition for equality remains the same. This is proved in Corollary \ref{cor:gammah} and used later to give an alternative proof of Theorem \ref{th:poissonuniversal}. In fact, the universality of the Poisson limit necessitates $\gamma(H)<|V(H)|-\nu(H)$ for all graphs with a cycle.

In a similar manner, the universal normal limit for $c\rightarrow\infty$ leads to the following interesting observation about $\gamma(H)$. Suppose $H$ has no isolated vertices: if $\gamma(H)> \frac{1}{2}|E(H)|$, then $H$ has a vertex of degree 1. This result is true for simple graphs as well as for multi graphs (with a similar definition of $\gamma$ for multi-graphs). This result is sharp, in the sense that there are simple graphs with no leaves such that $\gamma(H)=|E(H)|/2$. Even though this result follows easily from the
definition of $\gamma(H)$, it is a fortunate coincidence, as it is exactly what is needed in the proof of universal normality.

The role of cycle counts is crucial for the asymptotic normality when the number of colors is fixed. As the sequence $N(G_n)$ is uniformly integrable, the fourth moment condition and Theorem \ref{th:normal_fixed} imply the convergence of all other moments. In the language of graphs, this is equivalent to saying that the ACF4 condition implies that $N(G_n, C_g)=o(|E(G_n)|^{g/2})$, for all $g \geq 3$. This means that if number of 4-cycles in a graph is sub-extremal, then the counts of all other cycles are also sub-extremal. A combinatorial proof of this result and the similarities to results in pseudo-random graphs \cite{graham,conlonfoxzhao}, where the 4-cycle count plays a central role, are discussed in Section \ref{sec:4cycle}. 

\subsection{Other Monochromatic Subgraphs} The above theorems determine the universal asymptotic behavior of the number of monochromatic edges under independent and uniform coloring of the vertices. However, the situation 
for the number of other monochromatic subgraphs is quite different. Even under uniform coloring, the limit need not be a Poisson mixture. This is illustrated in the following proposition where we show that the number of monochromatic $r$-stars in a uniform coloring of an $n$-star converges to a polynomial in Poissons, which is not a Poisson mixture.
\begin{ppn}\label{propn:easy}
Let $G_n=K_{1,n}$, be the star graph with  $n+1$ vertices. Under the uniform coloring distribution, the random variable $T_{r, n}$ that counts the number of monochromatic $r$-stars in $G_n$ satisfies:
\[T_{r, n}\stackrel{\sD}{\rightarrow}\left\{
\begin{array}{ccc}
0  & \text{if}  &  \frac{n}{c}\rightarrow 0,\\
\infty & \text{if}  &  \frac{n}{c}\rightarrow \infty,\\
\frac{X(X-1)\cdots(X-r+1)}{r!}  & \text{if}  &  \frac{n}{c}\rightarrow \lambda,
\end{array}
\right.\]
where $X\sim Poisson(\lambda)$.
\end{ppn}
A few examples with other monochromatic subgraphs are also considered and several interesting observations are reported. We construct a graph $G_n$ where the number of monochromatic $g$-cycles ($g \geq 3$) in a uniform $c$-coloring of $G_n$ converges in distribution to a non-trivial mixture of Poisson variables even when $|N(G_n, C_g)|/c^{g-1}$ converges to a fixed number $\lambda$. This is in contrast to the situation for edges, where the number of monochromatic edges converges to $Poisson(\lambda)$ whenever $|E(G_n)|/c\rightarrow \lambda$. We believe that some sort of Poisson-mixture universality holds for cycles as well, that is, the number of monochromatic $g$-cycles in a uniform random coloring of any graph sequence $G_n$ converges in distribution to a random variable which is a mixture of Poisson, whenever $|N(G_n, C_g)|/c^{g-1}\rightarrow \lambda>0$.

\subsection{Literature Review on Non-Uniform Colorings}

A natural generalization of the uniform coloring distribution (\ref{eq:uniform}) is to consider a general coloring distribution $\vec p=(p_1, p_2, \ldots , p_c)$, that is, the probability a vertex is colored with color $a\in [c]$ is $p_a$ independent from the colors of the other vertices, where $p_a\ge 0$, and $\sum_{a=1}^c p_a=1$. Define $P_G(\vec p)$ to be the probability that $G$ is properly colored. $P_G(\vec p)$ is related to Stanley's generalized chromatic polynomial \cite{stanley}, and under the uniform coloring distribution it is precisely the proportion of proper $c$-colorings of
$G$. Recently, Fadnavis \cite{fadnavis} proved that $P_G(\vec p)$ is Schur-convex for every fixed $c$, whenever the graph $G$ is claw-free, that is, $G$ has no induced $K_{1,3}$. This implies that for claw-free graphs, the probability that it is properly colored is maximized under the uniform distribution, that is, $p_a=1/c$ for all $a\in [c]$. 

Poisson limit theorems for the number of monochromatic subgraphs in a random coloring of a graph sequence $G_n$ are applicable when the number of colors grows in an appropriate way compared to the number of certain specific subgraphs in $G_n$.  Arratia et al. \cite{arratia} used Stein's method based on dependency graphs to prove Poisson approximation theorems for the number of monochromatic cliques in a uniform coloring of a complete graph (see also Chatterjee et al. \cite{CDM}).  Poisson limit theorems for the number of general monochromatic subgraphs in a random coloring of a graph sequence are also given in Cerquetti and Fortini \cite{birthdayexchangeability}. They assumed that the distribution of colors was exchangeable and proved that the number of copies of any particular monochromatic subgraph converges in distribution to a mixture of Poissons. 

However, most of these results need conditions on the number of certain subgraphs in $G_n$ and the coloring distribution. Moreover, Poisson approximation holds only in the regime where the number of colors $c$ goes to infinity with $n$. Under the uniform coloring distribution, the random variables  in (\ref{eq:medges}) have more independence, and using the method of moments and estimates from extremal combinatorics, we show that nice universal limit theorems hold for the number of monochromatic edges. 

\subsection{Organization of the Paper}The rest of the paper is organized as follows: Section \ref{sec:mainproof} proves Theorem \ref{th:poissonuniversal} and Section \ref{sec:examplemain} illustrates it with various examples. Section \ref{sec:alon} discusses the connections with extremal combinatorics, fractional stable set polytope, and includes a new proof of Theorem \ref{th:alon_cycle}. The proofs of Theorem \ref{th:normal} and Theorem \ref{th:normal_fixed} are in Section \ref{sec:normal} and Section \ref{sec:normal_fixed}, respectively. The characterization of the limiting distribution for dense graphs (Theorem \ref{th:chisquare}) is detailed in Section \ref{sec:chisquare}. Finally, Section \ref{sec:starscycles} proves Proposition \ref{propn:easy}, considers other examples on counting monochromatic cycles, and discusses possible directions for future research. An appendix provides the details on conditional and unconditional convergence of random variables.

\section{Universal Poisson Approximation For Uniform Colorings: Proof of Theorem \ref{th:poissonuniversal}}
\label{sec:mainproof}

In this section, we determine the limiting behavior of $\P(N(G_n)=0)$ under minimal conditions. Using the method of moments, we show that $N(G_n)$ has a universal threshold, which depends only on the limiting behavior of $|E(G_n)|/c$, and a Poisson limit theorem holds at the threshold.

Let $G_n\in \sG_n$ be a random graph sampled according to some probability distribution. 
Recall the definition of $M(G_n)$ in (\ref{eq:independentbernoulli}). The proof of Theorem \ref{th:poissonuniversal} is given in two parts: The first part compares the conditional moments of $N(G_n)$ and $M(G_n)$, given the graphs $G_n$, showing that they are asymptotically close, when $|E(G_n)|/c\stackrel{\sD}\rightarrow Z$. The second part uses this result to complete the proof of Theorem \ref{th:poissonuniversal} using some technical properties of conditional convergence (see Lemma \ref{abc}).

\subsection{Computing and Comparing Moments}

This section is devoted to the computation of conditional moments of $N(G_n)$ and $M(G_n)$, and their comparison. To this end, define for any fixed number $k$, $A \lesssim_k B$ as $A\leq C(k) B$, where $C(k)$ is a constant that depends only on $k$. Let $G_n\in \sG_n$ be a random graph sampled according to some probability distribution. For any fixed subgraph $H$ of $G_n$, let $N(G_n, H)$ be the number of isomorphic copies of $H$ in $G_n$, that is,
$$N(G_n, H):=\sum_{S\subset E(G_n):|S|=|E(H)|}\pmb 1\{G_n[S]=H\},$$ where the sum is over subsets $S$ of $E(G_n)$ with $|S|=|E(H)|$, and $G_n[S]$ is the subgraph of $G_n$ induced by the edges of $S$.

\begin{lem}\label{third}Let $G_n\in \sG_n$ be a random graph sampled according to some probability distribution. For any $k \geq 1$, let $\cH_k$ be the collection of all graphs with at most $k$ edges and no isolated vertices. Then
\begin{eqnarray}
|\E(N(G_n)^k|G_n)-\E(M(G_n)^k|G_n)|&\lesssim_k&\sum_{\substack{H\in\cH_k,\\ H \text{ has a cycle}}}N(G_n, H)\cdot \frac{1}{c^{|V(H)|-\nu(H)}},
\label{eq:exp_diff}
\end{eqnarray}
where $\nu(H)$ is the number of connected components of $H$.
\label{lm:expectation_difference}
\end{lem}

\begin{proof}Using the multinomial expansion and the definition of $\cH_k$,
\begin{eqnarray}
\E(N(G_n)^k|G_n)&=&\sum_{(i_1, j_1)\in E(G_n)}\sum_{(i_2, j_2)\in E(G_n)}\cdots\sum_{(i_k, j_k)\in E(G_n)}\E\left(\prod_{r=1}^k\pmb1\{Y_{i_r}=Y_{j_r}\}\Big|G_n\right)
\label{eq:nk}
\end{eqnarray}
Similarly,
\begin{eqnarray}
\E(M(G_n)^k|G_n)=\sum_{(i_1, j_1)\in E(G_n)}\sum_{(i_2, j_2)\in E(G_n)}\cdots\sum_{(i_k, j_k)\in E(G_n)}\E\left(\prod_{r=1}^k Z_{i_rj_r}\right),
\label{eq:mk}
\end{eqnarray}
If $H$ is the simple subgraph of $G_n$ induced by the edges $(i_1, j_1), (i_2, j_2), \ldots (i_k, j_k)$. Then
$$\E\left(\prod_{r=1}^k\pmb1\{Y_{i_r}=Y_{j_r}\}\Big|G_n\right)=\frac{1}{c^{|V(H)|-\nu(H)}}\text{ and } \E\left(\prod_{r=1}^k Z_{i_rj_r}\right)=\frac{1}{c^{|E(H)|}}.$$
The result now follows by taking the difference of (\ref{eq:nk}) and (\ref{eq:mk}), and noting that in any graph $H$, $|E(H)|\geq |V(H)|-\nu(H)$ and equality holds if and only if $H$ is a forest.
\end{proof}


Lemma \ref{lm:expectation_difference} shows that bounding the difference of the conditional moments of $N(G_n)$ and $M(G_n)$ entails bounding $N(G_n, H)$, for all graphs $H$ with a cycle. The next lemma estimates the number of copies of a cycle $C_g$ in $G_n$.

\begin{lem}\label{cycle}
For $g\ge 3$ and $G_n\in \sG_n$ let $N(G_n, C_g)$ be the number of $g$-cycles in $G_n$. Then $$N(G_n, C_g)\le \frac{1}{2g}\cdot(2|E(G_n)|)^{g/2}.$$
\end{lem}

\begin{proof}Let $A:=A(G_n)$ be the adjacency matrix of $G_n$. Note that $\sum_{i=1}^n\lambda_i^2(G_n)=\tr(A^2)=2|E(G_n)|$, where $\vec\lambda(G_n)=(\lambda_1(G_n), \cdots \lambda_n(G_n))'$ is the vector of eigenvalues of $A(G_n)$. Note that $\tr(A^g)$ counts the number of walks of length $g$ in $G_n$, and so each cycle in $G_n$ is counted $2g$ times. Thus, for any $g\geq 3$, 
$$N(G_n, C_g)\le \frac{1}{2g}\cdot\tr(A^g)=\frac{1}{2g}\cdot\sum_{i=1}^n\lambda_i^g(G_n)
\leq\frac{1}{2g}\cdot ||\vec\lambda(G_n)||_\infty^{g-2}\sum_{i=1}^n\lambda_i^2(G_n)\leq \frac{1}{2g}\cdot(2|E(G_n)|)^{g/2},$$
where the last step uses $||\vec\lambda(G_n)||_\infty^{g-2}\leq (2|E(G_n)|)^{g/2-1}$.
\end{proof}


For $a, b\in \R$, $a\lesssim b$, $a\gtrsim b$, and $a\asymp b$ means $a\leq C_1 b$, $a\geq C_2 b$ and
$C_2 b\leq a\leq C_1 b$ for some constants $C_1, C_2\geq 0$, respectively. For a given simple graph $H$, the notation $A \lesssim_H B$ will mean $A\leq C(H)\cdot B$, where $C(H)$ is a constant that depends only on $H$. The following lemma gives a bound on $N(G_n, H)$ in terms of $|E(G_n)|$ for arbitrary subgraphs $H$ of $G_n$.

\begin{lem}\label{fourth}For any  fixed connected subgraph $H$,  let $N(G_n, H)$ be the set of copies of $H$ in $G_n$. Then
\begin{equation}
N(G_n, H)\lesssim_H |E(G_n)|^{|V(H)|-1}.
\label{eq:subgraphcount}
\end{equation}
 Furthermore, if $H$ has a cycle of length $g\geq 3$, then
\begin{equation}
N(G_n, H)\lesssim_H |E(G_n)|^{|V(H)|-g/2}.
\label{eq:subgraphcountcycle}
\end{equation}
\label{lm:count}
\end{lem}

\begin{proof}The first bound on $N(G_n, H)$ can be obtained by a crude counting argument as follows: First choose an edge of $G_n$ in $E(G_n)$ which fixes 2 vertices of $H$. Then  the remaining $|V(H)|-2$ vertices are chosen arbitrarily from the set of allowed $V(G_n)$ vertices, giving the bound $$
N(G_n, H)\lesssim_H |E(G_n)|{|V(G_n)|\choose |V(H)|-2}\leq  |E(G_n)|{2|E(G_n)|\choose |V(H)|-2}\lesssim_H |E(G_n)|^{|V(H)|-1},$$  where we have used the fact that the number of graphs on $|V(H)|$ vertices is at most  $2^{{|V(H)|\choose 2}}$.

Next, suppose that $H$ has a cycle of length $g\geq 3$. Choosing a cycle of length $g$ arbitrarily from $G_n$, there are $|V(G_n)|$ vertices from which the remaining $|V(H)|-g$ vertices are chosen arbitrarily. Since the edges among these vertices are also chosen arbitrarily, the following crude upper bound holds
\begin{eqnarray}
N(G_n, H)\lesssim_H N(G_n, C_g){{2|E(G_n)|}\choose{|V(H)|-g}}&\lesssim_H &N(G_n, C_g)|E(G_n)|^{|V(H)|-g}\nonumber\\
&\lesssim_H &|E(G_n)|^{|V(H)|-g/2}.
\label{eq:cyclecalc}
\end{eqnarray}
where the last step uses Lemma \ref{cycle}.
\end{proof}

The above lemmas, imply the most central result of this section: the conditional moments of $M(G_n)$ and $N(G_n)$ are asymptotically close, whenever $|E(G_n)|/c\stackrel{\sD}{\rightarrow} Z$.

\begin{lem}Let $M(G_n)$ and $N(G_n)$ be as defined in (\ref{eq:medges}) and (\ref{eq:independentbernoulli}), with  $|E(G_n)|/c\stackrel{\sD}{\rightarrow} Z$, then for every fixed $k\geq 1$, $$|\E (N(G_n)^k|G_n)-\E (M(G_n)^k|G_n)|\stackrel{\sP}{\rightarrow}0.$$
\label{lm:difference_zero}
\end{lem}

\begin{proof}By Lemma \ref{third}
$$|\E(N(G_n)^k|G_n)-\E(M(G_n)^k|G_n)|\leq \sum_{\substack{H\in\cH_k,\\ H \text{ has a cycle}}}N(G_n, H)\cdot \frac{1}{c^{|V(H)|-\nu(H)}},$$
where $\nu(H)$ is the number of connected components of $H$. As the sum over $H\in\cH_k$ is a finite sum, it suffices to show that for a given $H\in \cH_k$ with a cycle $N(G_n, H)=o_P(c^{|V(H)|-\nu(H)}).$

To this end, fix $H\in \cH_k$ and let $H_1, H_2, \ldots, H_{\nu(H)}$ be the connected components of $H$. Without loss of generality, suppose the girth of $G$, $g(G)=g(H_1)=g \geq 3$. Lemma \ref{lm:count} then implies that
\begin{eqnarray}
N(G_n, H)\le \prod_{i=1}^{\nu(H)}|N(G_n, H_i)\lesssim_H & |E(G_n)|^{|V(H_1)|-g/2}\prod_{i=2}^{\nu(H)}E(G_n)^{V(H_i)-1}\nonumber\\
\lesssim_H &|E(G_n)|^{|V(H)|-\nu(H)+1-g/2}
\label{eq:exp_cycle}
\end{eqnarray}
which is $o_p(c^{|V(H)|-\nu(H)})$ since $g/2-1>0$.
\end{proof}


\subsection{Completing the Proof of Theorem \ref{th:birthdayuniformuniversal}}
The results from the previous section are used here to complete the proof of Theorem \ref{th:birthdayuniformuniversal}. The three different regimes of $|E(G_n)|/c$ are treated separately as follows:

\subsubsection{$\frac{1}{c}\cdot |E(G_n)|\stackrel{\sP}{\rightarrow} 0$}

 In this case, $\P(N(G_n)>0|G_n)\leq \E(N(G_n)|G_n)=|E(G_n)|/c\stackrel{\sP}{\rightarrow} 0$, and the result follows. 

\subsubsection{$\frac{1}{c}\cdot |E(G_n)|\stackrel{\sP}{\rightarrow} \infty$} 

Since $\E(N(G_n)^2|G_n)=\frac{|E(G_n)|^2}{c^2}+\frac{|E(G_n)|}{c}$, in this case $$\frac{\E(N(G_n)^2|G_n)}{\E(N(G_n)|G_n)^2}\stackrel{\sP}{\rightarrow} 1.$$ This implies that  $N(G_n)/\E(N(G_n)|G_n)$ converges in probability to 1, and so
$N(G_n)$ converges to $\infty$ in probability, as $\E(N(G_n)|G_n)=\frac{1}{c}\cdot |E(G_n)|\stackrel{\sP}\rightarrow \infty$.

\subsubsection{$\frac{1}{c}\cdot |E(G_n)|\stackrel{\sD}{\rightarrow} Z$, where $Z$ is some random variable} \label{sec:proofcompletetechnical}

In this regime the limiting distribution of $N(G_n)$ is a mixture of Poisson. As the Poisson distribution can be uniquely identified by moments, from Lemma \ref{lm:difference_zero} it follows that
conditional on $\{|E(G_n)|/c\rightarrow\lambda\}$, $N(G_n)$ converges to $Poisson(\lambda)$ for every $\lambda>0$. However, this does not immediately imply the unconditional convergence of $N(G_n)$ to a mixture of Poisson. In fact, a technical result, detailed in  Lemma \ref{abc}, and convergence of $M(G_n)$ to a Poisson mixture is necessary to complete the proof.

To begin with, recall that a random variable $X$ is a mixture of Poisson with mean $Z$, to be denoted as $Poisson(Z)$, if there exists a non-negative random variable $Z$ such that $$\P(X=k)=\E\left(\frac{1}{k!} e^{-Z}Z ^k\right).$$
The following lemma shows that $M(G_n)$ converges to $Poisson (Z)$ and satisfies the technical condition needed to apply Lemma \ref{abc}.

\begin{lem}
\label{lm:obvious}
Let $M(G_n)$ be as defined in (\ref{eq:independentbernoulli}) and $\frac{1}{c}\cdot |E(G_n)|\stackrel{\sD}{\rightarrow} Z$. Then $M(G_n)\stackrel{\sD}{\rightarrow}Poisson(Z)$, and further for any $\varepsilon>0,t\in \R$,
\begin{align*}
\limsup_{k\rightarrow\infty}\limsup_{n\rightarrow \infty}\P\left(\left|\frac{t^k}{k!}\E(M(G_n)^k|G_n)\right|>\varepsilon\right)=0.
\end{align*}
\end{lem}

\begin{proof}
For any $t\in \mathbb R$, $\E e^{itM(G_n)}=\E \E (e^{itM(G_n)}|G_n)=\E \left(1-\frac{1}{c}+\frac{e^{it}}{c}\right)^{|E(G_n)|}=\E R_n$,
where $R_n:=\left(1-\frac{1}{c}+\frac{e^{it}}{c}\right)^{|E(G_n)|}$ satisfies $|R_n|\le 1$. Since
$$\log R_n=|E(G_n)|\log \left(1-\frac{1}{c}+\frac{e^{it}}{c}\right)= |E(G_n)|\left(\frac{e^{it}-1}{c}+O\left(\frac{1}{c^2}\right)\right)\stackrel{\sD}{\rightarrow}(e^{it}-1)Z,$$
by the dominated convergence theorem
$\E e^{itM(G_n)}=\E R_n\rightarrow \E e^{(e^{it}-1)Z}$, which can be easily checked to be the generating function of a random variable with distribution $Poisson(Z)$. Thus, it follows that $M(G_n)\stackrel{\sD}{\rightarrow} Poisson(Z)$.

Proceeding to check the second conclusion, recall the standard identity $z^k=\sum_{j=0}^k S(k, j)(z)_j$, where $S(\cdot,\cdot)$ are Stirling numbers of the second kind and $(z)_j=z(z-1)\cdots(z-j+1)$. In the above identity, setting $z=M(G_n)$, taking expectation on both sides conditional on $G_n$, and using the formula for the Binomial factorial moments,
\begin{align*}
\E (M(G_n)^k|G_n)=\sum_{j=0}^k S(k,j)(|E(G_n)|)_jc^{-j}.
\end{align*}
The right hand side converges weakly to $\sum_{j=0}^k S(k,j) Z^j$. This is the $k$-th mean of a Poisson random variable with parameter $Z$. Using the formula for the Poisson moment generating function, for any $Z\ge 0$  and any $t\in \R$ we have
$$\sum_{k=0}^\infty \frac{t^k}{k!}\sum_{j=0}^k S(k,j)Z^j=e^{Z(e^t-1)}<\infty \implies
\frac{t^k}{k!}\sum_{j=0}^k S(k,j) Z^j\stackrel{a.s.}{\rightarrow}0,$$ as $k\rightarrow\infty$. Thus, applying Fatou's Lemma 
\begin{align*}
\limsup_{n\rightarrow \infty}\P\left(\Big|\frac{t^k}{k!}\E(M(G_n)^k|G_n)\Big|>\varepsilon\right)\le \P\left(\Big|\frac{t^k}{k!}\sum_{r=0}^k S(k,r)Z^r\Big|>\varepsilon\right).
\end{align*}
The lemma now follows by taking limit as $k\rightarrow \infty$ on both sides.
\end{proof}

Now, take $X_n=M(G_n)$ and $Y_n=N(G_n)$, and observe that (\ref{eq:abc1}) and (\ref{eq:abc2}) hold by Lemma \ref{lm:difference_zero} and Lemma \ref{lm:obvious}, respectively. As $M(G_n)$ converges to $Poisson(Z)$, this implies that $N(G_n)$ converges to $Poisson(Z)$, and the proof of Theorem \ref{th:birthdayuniformuniversal} is completed.


\begin{remark}Theorem \ref{th:birthdayuniformuniversal} shows that the limiting distribution of the number of monochromatic edges converges to a Poisson mixture. In fact, Poisson mixtures arise quite naturally in several contexts. It is known that the Negative Binomial distribution is distributed as $Poisson(Z)$, where $Z$ is a Gamma random variable with integer values for the shape parameter.  Greenwood and Yule \cite{greenwoodyule} showed that certain empirical distributions of accidents are well-approximated by a Poisson mixture. Le-Cam and Traxler \cite{lecam} proved asymptotic properties of random variables distributed as a mixture of Poisson. Poisson mixtures are widely used in modeling count panel data (refer to the recent paper of Burda et al. \cite{poissonmixturediscretechoice} and the references therein),  and have appeared in other applied problems as well \cite{poissonmixture}.
\end{remark}

\section{Examples: Applications of Theorem \ref{th:birthdayuniformuniversal}}
\label{sec:examplemain}

In this section we apply Theorem \ref{th:birthdayuniformuniversal} to different deterministic and random graph models, and determine the specific nature of the limiting Poisson distribution.

	
\begin{example}(Birthday Problem)
When the underlying graph $G$ is the complete graph $K_n$ on $n$ vertices, the above coloring problem reduces to the well-known birthday problem. By replacing the $c$ colors by birthdays, each occurring with probability $1/c$, the birthday problem can be seen as coloring the vertices of a complete graph independently with $c$ colors. The event that two people share the same birthday is the event of having a monochromatic edge in the colored graph. In birthday terms, $\P(N(K_n)=0)$ is precisely the probability that no two people have the same birthday.
Theorem \ref{th:birthdayuniformuniversal} says that under the uniform coloring for the complete graph
$\mathbb{P}(N(K_n)=0)\approx e^{-n^2/2c}$. Therefore, the maximum $n$ for which $\P(N(K_n)=0)\leq 1/2$ is approximately 23, whenever $c=365$. This reconstructs the classical birthday problem which can also be easily proved by elementary calculations. For a detailed discussion on the birthday problem and its various generalizations and applications, refer to \cite{aldous,barbourholstjanson,dasguptasurvey,diaconisholmes,diaconismosteller} and the references therein.
\end{example}

\begin{example}(Birthday Coincidences in the US Population) Consider the following question: What is the chance that there are two people in the United States who (a) know each other, (b) have the same birthday,  (c) their fathers have the same birthday,  (d) their grandfathers have the same birthday, and (e) their great grandfathers have the same birthdays. We will argue that this seemingly impossible coincidence actually happens with almost absolute certainty.

The population of the US is about $n$=400 million and it is claimed that a typical person knows about 600 people \cite{gelman,killworthbernard}. If the network $G_n$ of `who knows who' is modeled as an Erd\H os-Renyi graph, this gives $p=150 \times 10^{-8}$ and $\E(|E(G_n)|)=300\times 4\times 10^8=1.2\times 10^{11}$. The 4-fold birthday coincidence amounts to $c=(365)^4$ `colors' and $\lambda=\E(N(G_n))=\E(|E(G_n)|)/c\approx 6.76$, and using
the bound $\P(N(G_n)>0)\geq \frac{\E(N(G_n))^2}{\E(N(G_n)^2)}$, the probability of a match is at least $1-\frac{1}{\lambda}=85\%$. Note that this bound only uses an estimate on the number of edges in the graph. Moreover, assuming the Poisson approximation, the chance of a match is approximately $1-e^{-\lambda}\approx 99.8\%$, which means that almost surely there are two friends in the US who have a 4-fold birthday match among their ancestors.

Going back one more generation, we now calculate the probability that there are two friends who have a 5-fold birthday coincidence between their respective ancestors. This amounts to $c=(365)^5$  and Poisson approximation shows that the chance of a match is approximately $1-e^{-\lambda}\approx 1.8\%$. This implies that even a miraculous 5-fold coincidence of birthdays is actually likely to happen among the people of the US.

This is an example of the {\it law of truly large numbers} \cite{diaconismosteller}, which says that when enormous numbers of events and people and their interactions cumulate over time, almost any outrageous event is bound to occur. The point is that truly rare events are bound to be plentiful in a population of 400 million people. If a coincidence occurs to one person in a million each day, then
we expect 400 occurrences a day and close to 140,000 such occurrences a year. 
\end{example}

%
%
%
%

\begin{example}(Galton-Watson Trees)
An example where the limiting distribution of $N(G_n)$ is indeed a Poisson mixture arises in the uniform coloring of Galton-Watson trees with general offspring distribution. Let $G_n$ be a Galton-Watson tree truncated at height $n$, and let $\xi$ denote a generic random variable from the offspring distribution. Assume further that  $\mu:=\E\xi>1$. This ensures that the total progeny up
to time $n$ (which is also the number of edges in $G_n$) grows with $n$. Letting $\{S_i\}_{i=0}^\infty$ denote the size of the $i$-th generation, the total progeny up to time $n$ is $Y_n:=\sum_{i=0}^nS_i$. Assuming that the population starts with one off-spring at time $0$, that is, $S_0\equiv 1$, $S_n/\mu^n$ is a non-negative martingale (\cite[Lemma 4.3.6]{durrett}). It converges almost surely to a finite valued random variable $S_\infty$, by \cite[Theorem 4.2.9]{durrett}, which readily implies $Y_n/\mu^{n+1}$ converges almost surely to $S_\infty/(\mu-1)$. Thus,  Theorem \ref{th:birthdayuniformuniversal}  gives
$$
\P(N(G_n)=0)\rightarrow \left\{
\begin{array}{ccc}
1  & \text{if}  &  \frac{\mu^n}{c} \rightarrow 0,\\
0  & \text{if}  &   \frac{\mu^n}{c}\rightarrow \infty, \\
\E e^{-\frac{b\mu}{\mu-1}\cdot S_\infty}  & \text{if}  &   \frac{\mu^n}{c}\rightarrow b.
\end{array}
\right.
$$
Note in passing that $S_\infty\equiv 0$ if and only if $\E(\xi\log \xi)=\infty$ (\cite[Theorem 4.3.10]{durrett}). Thus, to get a nontrivial limit  the necessary and sufficient condition is $\E(\xi\log \xi)<\infty$.

\end{example}

\section{Connections to Extremal Graph Theory}
\label{sec:alon}

In the method of moment calculations of Lemma \ref{lm:expectation_difference}, we encounter the quantity $N(G, H)$, the number of isomorphic copies of $H$ in $G$.
More formally, given two graphs $G=(V(G), E(G))$ and $H=(V(H), E(H))$, we have
$$N(G, H)=\sum_{S\subset E(G):|S|=|E(H)|}\pmb 1\{G[S]=H\},$$ where the sum is over subsets $S$ of $E(G)$ with $|S|=|E(H)|$, and $G[S]$ is the subgraph of $G$ induced by the edges of $S$.

For a positive integer $\ell\geq |E(H)|$, define $N(\ell, H):=\sup_{G: |E(G)|=\ell}N(G, H)$.
For the complete graph $K_h$, Erd\H os \cite{erdoscompletesubgraphs} determined $N(\ell, K_h)$, which is also a special case of the Kruskal-Katona theorem, and posed the problem of estimating $N(\ell, H)$ for other graphs $H$. This was addressed by Alon \cite{alon81} in 1981 in his first published paper. Alon studied  the asymptotic behavior of $N(\ell, H)$ for fixed $H$, as $\ell$ tends to infinity. He identified the correct order of $N(\ell, H)$, for every fixed $H$, by proving that:
\begin{thm}[Alon \cite{alon81}]For a fixed graph $H$, there exist two positive constants $c_1=c_1(H)$ and $c_2=c_2(H)$ such that for all $\ell\geq |E(H)|$,
\begin{equation}
c_1\ell^{\gamma(H)}\leq N(\ell, H) \leq c_2\ell^{\gamma(H)},
\label{eq:alon_exponent}
\end{equation}
where $\gamma(H)=\frac{1}{2}(|V(H)|+\delta(H))$, and $\delta(H)=\max\{|S|-|N_H(S)|: S\subset V(H)\}$.
\label{th:alon_exponent}
\end{thm}

Hereafter, unless mentioned otherwise, we shall only consider graphs $H$ with no isolated vertex. Using the above theorem or the definition of $\gamma(H)$ it is easy to show that $\gamma(H)\leq |E(H)|$, and the equality holds if and only if $H$ is a disjoint union of stars (Theorem 1, Alon \cite{alon86}). The following corollary gives a sharpening of Theorem 1 of \cite{alon86}:

\begin{cor}For every graph $H$, $$\gamma(H)\leq |V(H)|-\nu(H),$$ where $\nu(H)$ is the number of connected components of $H$. Moreover, the equality holds if and only if $H$ is a disjoint union of stars.
\label{cor:gammah}
\end{cor}

\begin{proof}
Suppose $H_1, H_2, \ldots, H_{\nu(H)}$ are the connected components of $H$. Fix $i \in \{1,2, \ldots, \nu(H)\}$. Since $H_i$ is connected, for every $S\subset V(H_i)$,
$|S|-|N_{H_i}(S)|\leq |V(H_i)|-2$. This implies that $\delta(H)=\sum_{i=1}^{\nu(H)}\delta(H_i)\leq |V(H)|-2\nu(H)$, and $\gamma(H)\leq |V(H)|-\nu(H)$.

Now, if $H$ is a disjoint union of stars with $\nu(H)$ connected components, then by Theorem 1 of Alon \cite{alon86}, $\gamma(H)=|E(H)|=|V(H)|-\nu(H)$.

Conversely, suppose that $\gamma(H)=|V(H)|-\nu(H)$. If $H$ has a cycle of length $g \geq 3$, then from (\ref{eq:exp_cycle}) and Theorem \ref{th:alon_exponent} $\gamma(H)\leq |V(H)|-\nu(H)+1-g/2<|V(H)|-\nu(H)$. Therefore, $H$ has no cycle, that is, it is a disjoint union of trees.  This implies that $\gamma(H)=|V(H)|-\nu(H)=|E(H)|$, and from Theorem 1 of Alon \cite{alon86}, $H$ is a disjoint union of stars.
\end{proof}

\subsection{Another Proof of Theorem \ref{th:birthdayuniformuniversal}}

Theorem \ref{th:alon_exponent} and Corollary \ref{cor:gammah} give a direct proof of Lemma \ref{lm:difference_zero}, which does not require the subgraph counting Lemmas \ref{cycle} and \ref{fourth}.

With  $N(G_n)$ and $M(G_n)$  as defined before in (\ref{eq:medges}) and (\ref{eq:independentbernoulli}), and $|E(G_n)|/c\stackrel{\sD}{\rightarrow} Z$, for every fixed $k\geq 1$ we have
\begin{equation*}
|\E(N(G_n)^k|G_n)-\E(M(G_n)^k|G_n)|\lesssim_k \sum_{\substack{H\in\cH_k,\\ H \text{ has a cycle}}}N(G_n, H)\cdot \frac{1}{c^{|V(H)|-\nu(H)}}\lesssim_k \sum_{\substack{H\in\cH_k,\\ H \text{ has a cycle}}}\cdot\frac{|E(G_n)|^{\gamma(H)}}{c^{|V(H)|-\nu(H)}},
\end{equation*}
where the last inequality follows from Theorem \ref{th:alon_exponent}. As the sum is over all graphs $H$ that are not a forest, it follows from Corollary \ref{cor:gammah} that $\gamma(H)< |V(H)|-\nu(H)$. Therefore, every term in the sum goes to zero as $n\rightarrow \infty$, and, since $H\in\cH_k$ is a finite sum, Lemma \ref{lm:difference_zero} follows.


\subsection{A New Proof of Theorem \ref{th:alon_cycle} Using Lemma \ref{cycle}} This section gives a short proof of Theorem \ref{th:alon_cycle} using Lemma \ref{cycle}.

\subsubsection{Proof of Theorem \ref{th:alon_cycle}} Let $F$ be the spanning subgraph $H$, and let $F_1, F_2, \ldots, F_q$, be the connected components of $F$, where each $F_i$ is a cycle or an isolated edge, for $i \in \{1,2, \ldots, q\}$. Consider the following two cases:

\begin{description}
\item[{\it Case} 1]$F_i$ is an isolated edge. Then for any graph $G$ with $|E(G)|=\ell$,
\begin{equation}
N(G, F_i)=\ell=\frac{1}{|Aut(F_i)|}\cdot(2\ell)^{|V(F_i)|/2}.
\label{eq:isolatededge}
\end{equation}
\item[{\it Case} 2]$F_i$ is a cycle of length $g\geq 3$. Then by Lemma \ref{cycle}
\begin{equation}
N(G, F_i)\leq \frac{1}{2g}\cdot (2\ell)^{g/2}=\frac{1}{|Aut(F_i)|}\cdot(2\ell)^{|V(F_i)|/2},\label{eq:cycle}
\end{equation}
for any graph $G$ with $|E(G)|=\ell$.
\end{description}
Now, (\ref{eq:isolatededge}) and (\ref{eq:cycle}) imply that
\begin{equation}
N(G, F)\leq \prod_{i=1}^q N(G, F_i)\leq \frac{1}{\prod_{i=1}^q |Aut(F_i)|}\cdot(2\ell)^{|V(H)|/2}=\frac{1}{|Aut(F)|}\cdot(2\ell)^{|V(H)|/2},
\label{eq:isolatededge_cycle}
\end{equation}
for all graphs $G$ with $|E(G)|=\ell$.

Let $v=|V(H)|=|V(F)|$ and define $x(H, F)$ to be the number of subgraphs of $K_v$, isomorphic to $H$, that contain a fixed copy of $F$ in $K_v$. Given a graph $G$ with $|E(G)|=\ell$, every $F$ in $G$ can be completed (by adding edges) to an $H$ in $G$, in at most $x(H,F)$ ways, and in this fashion each $H$ in $G$ is obtained exactly $N(H, F)$ times (see \cite[Lemma 3]{alon81}). This implies that
\begin{equation}
N(\ell, H)\leq \frac{x(H, F)}{N(H, F)}N(\ell, F)
\label{eq:alon_bound_I}
\end{equation}
Similarly, $N(K_{v}, H)=\frac{x(H, F)}{N(H, F)}N(K_{v}, F)$ (see \cite[Lemma 6]{alon81}) and it follows from (\ref{eq:alon_bound_I}) that,
\begin{eqnarray}
N(\ell, H)\leq \frac{N(K_{v}, H)}{N(K_{v}, F)}N(\ell, F)=\frac{|Aut(F)|}{|Aut(H)|}N(\ell, F)\leq  \frac{1}{|Aut(H)|}(2\ell)^{|V(H)|/2},
\end{eqnarray}
where the last inequality follows from (\ref{eq:isolatededge_cycle}).

For the lower bound, let $s=\lfloor \sqrt{2\ell}\rfloor$ and note that,
\begin{eqnarray}
N(\ell, H)\ge N(K_{s}, H)&=&{s\choose |V(H)|}N(K_{|V(H)|}, H)\nonumber\\
&=&\frac{s^{|V(H)|}+O(s^{|V(H)|-1})}{|V(H)|!}N(K_{|V(H)|}, H)\nonumber\\
&=&\frac{(2\ell)^{|V(H)|/2}}{|Aut(H)|}+O(\ell^{|V(H)|/2-1/2}),\nonumber
\end{eqnarray}
thus completing the proof. \hfill $\Box$


\subsection{Connections to Fractional Stable Set and a Structural Lemma}

Friedgut and Kahn \cite{hypergraphcopies} extended Alon's result to hypergraphs, and identified the exponent $\gamma(H)$ as the fractional stable number of the hypergraph $H$, which is the defined as the solution of a linear programming problem. Using this alternative definition, we can define $\gamma(H)$ for any multigraph as follows:
\begin{equation}
\gamma(H)=\arg\max_{\phi\in V_{H}[0,1]} \sum_{v\in V(H)} \phi(v) \text{ subject to } \phi(x)+\phi(y)\le 1 \text{ for }(x,y)\in E(H),
\label{eq:gamma}
\end{equation}
where $V_{H}[0,1]$ is the collection of all functions $\phi: V(H)\rightarrow [0, 1]$. It is clear  that $\gamma(H)=\gamma(H_S)$, where $H_S$ is the simple graph obtained from $H$ by replacing multiple edges with single edges. The polytope $\sP_H$ defined by the constraint set of this linear program is called the {\it fractional stable set polytope} \cite{schrijver}. 

A discussion on the various useful properties of the fractional independence number $\gamma(H)$ can be found in Janson et al. \cite{janson_subgraph}. The following lemma, which is a easy consequence of the definitions, relates $\gamma(H)$ to the minimum degree of $H$, to be denoted by $d_{\min}(H)$.  As before, we will only be considering multigraphs with no isolated vertex.

\begin{lem}
Let $H=(V(H), E(H))$ be a multigraph with no isolated vertex and $d_{\min}(H)\geq 2$, and let $\varphi:V(H)\rightarrow [0, 1]$ be an optimal solution of the linear program (\ref{eq:gamma}). Then $\gamma(H)\leq \frac{1}{2}|E(H)|$. Moreover, if there exists $v\in V(H)$ such that $\varphi(v)\ne 0$ and $d(v)\geq 3$ then $\gamma(H)< \frac{1}{2}|E(H)|$.
\label{lm:degree_one}
\end{lem}

\begin{proof}Suppose $d_{\min}(H)\geq 2$. Then for any
$\phi: V(H)\rightarrow [0, 1]$ such that  $\phi(x)+\phi(y)\le 1$ for $(x,y)\in E(H)$,
$$\sum_{x\in V(H)} \phi(x)\leq \frac{1}{d_{\min}(H)}\sum_{(x, y)\in E(H)}\{\phi(x)+\phi(y)\}\leq\frac{1}{2}|E(H)|.$$
Taking the maximum over functions $\phi$, it follows that $\gamma(H)\leq \frac{1}{2}|E(H)|$.

Now, suppose there exists $v\in V(H)$ such that $\varphi(v)\ne 0$ and $d(v)\geq 3$. Then
\begin{eqnarray}
|E(H)|\geq \sum_{(x, y)\in E(H)}\{\varphi(x)+\varphi(y)\}=\sum_{x\in V(H)} d(x)\varphi(x) \geq 3\varphi(v)+2\sum_{x\in V(H)\backslash \{v\}}\varphi(x)=2\gamma(H)+\varphi(v),\nonumber
\end{eqnarray}
and the result follows since $\varphi(v)>0$.
\end{proof}

The fractional stable set polytope $\sP(H)$ is a widely studied object in combinatorial optimization \cite{schrijver}. One of the most important property of the fractional stable set polytope is the following:

\begin{ppn}[\cite{nt_I}] Let $\phi: V(H)\rightarrow [0, 1]$ be any extreme point of the fractional stable set polytope $\sP(H)$. Then $\phi(v)\in \{0, \frac{1}{2}, 1\}$, for all $v\in V(H)$.
\end{ppn}

Let $H$ be any multi-graph with no isolated vertex, and $\varphi:V(H)\rightarrow [0, 1]$ be an extreme point of $\sP(H)$ that is the optimal solution to the linear program defined in (\ref{eq:gamma}). If $\varphi(v)=1/2$ for all $v\in V(H)$, then $\gamma(H)=|V(H)|/2$ and by Alon \cite[Theorem 4]{alon81}, $H$ is a disjoint union of cycles or isolated edges. When $\gamma(H)> |V(H)|/2$, the above proposition can be used to prove a structural result for the optimal function $\varphi$. To this end, partition $V(H)=V_0(H)\cup V_{1/2}(H)\cup V_1(H)$, where $V_a(H)=\{v\in V(H): \varphi(v)=a\}$, for $a\in \{0, 1/2, 1\}$. Note that $V_1(H)$ must be an independent set in $H$. The following lemma gives structural properties of the subgraphs of $H$ induced by this partition of the vertex set.  The proof closely follows the proof of Lemma 9 of Alon \cite{alon81}, but is re-formulated here in terms of the function $\varphi$. The lemma will be used later to prove the normality of $N(G_n)$ in Theorem \ref{th:normal_fixed}.

\begin{lem}Let $H$ be a multigraph with no isolated vertex and $\gamma(H)> |V(H)|/2$. If $\varphi:V(H)\rightarrow [0, 1]$ is an optimal solution to the linear program (\ref{eq:gamma}), then the following holds:
\begin{enumerate}
\item[(i)] The bipartite graph $H_{01}=(V_0(H)\cup V_1(H), E(H_{01}))$, where $E(H_{01})$ is set of edges from $V_0(H)$ to $V_1(H)$, has a matching which saturates every vertex in $V_0(H)$.

\item[(ii)] The subgraph of $H$ induced by the vertices of $V_{1/2}(H)$ has a spanning subgraph which is a disjoint union of cycles and isolated edges.
\end{enumerate}
\label{lm:gamma_structural}
\end{lem}

\begin{proof} By Hall's marriage theorem, the bipartite graph $H_{01}$ has a matching that saturates every vertex in $V_0(H)$, if and only if for all $S \subset V_0(H)$, $|N_{H_{01}}(S)|\geq |S|$. Suppose this is false and there exists $A \subset V_0(H)$ such that $|N_{H_{01}}(A)|< |A|$.
Let $B=V_1(H)\backslash N_{H_{01}}(A)$. By assumption, $\gamma(H)=\frac{1}{2}|V_{1/2}(H)|+|V_1(H)|> \frac{1}{2}(|V_0(H)|+|V_{1/2}(H)|+|V_1(H)|)$, which implies that $|V_1(H)|> |V_0(H)|$, and so
$$|B|=|V_1(H)|-|N_{H_{01}}(A)|>|V_1(H)|-|A|\geq |V_1(H)|-|V_0(H)|>0.$$ 
In particular, this implies that the set $B$ is not empty.

Now, define a function $\tilde\varphi:V(H)\rightarrow [0, 1]$ such that
$$
\tilde\varphi(v):=\left\{
\begin{array}{ccc}
1  & \text{ if }  v \in B, \\
0  & \text{ if }   v \in N_H(B), \\
\frac{1}{2}  & \text{otherwise}.
\end{array}
\right.
$$
Note that $\tilde\phi\in \sP(H)$, and $\sum_{v\in V}\tilde\varphi(v)=\frac{1}{2}(|V(H)|+|B|-|N_H(B)|)$. Now, as $N_H(B)\subseteq V_0(H)\backslash A$,
\begin{eqnarray}
\sum_{v\in V}\tilde\varphi(v)=\frac{1}{2}(|V(H)|+|B|-|N_H(B)|)&=&\frac{1}{2}(|V(H)|+|V_1(H)|-|N_{H_{01}}(A)|-|N_H(B)|)\nonumber\\
&\geq &\frac{1}{2}(|V(H)|+|V_1(H)|-|V_0(H)|-|N_{H_{01}}(A)|+|A|)\nonumber\\
&>&\frac{1}{2}(|V(H)|+|V_1(H)|-|V_0(H)|)\nonumber\\
&=&\frac{1}{2}|V_{1/2}(H)|+|V_1(H)|=\gamma(H).
\end{eqnarray}
This contradicts the maximality of $\gamma(H)$ and proves that $H_{01}$ has a matching that saturates every vertex in $V_0(H)$.

Next, denote the subgraph of $H$ induced by the vertices of $V_{1/2}(H)$  by $F$. To prove that $F$ has a spanning subgraph which is a disjoint union of cycles and isolated edges, it suffices to show that for all $S\subset V(F)$, $|N_{F}(S)|\geq |S|$ (Lemma 7, Alon \cite{alon81}). Assuming this is false, there exists $C\subset V(F)$, such that $C\cap N_H(C)=\emptyset$ and $|N_{F}(C)|< |C|$. Define $D=C\cup V_1(H)$. Note that $N_H(D)=N_H(C)\cup N_H(V_1(H))=N_{F}(C)\cup N_H(V_1(H))$
$$|D|-|N_H(D)|=|V_1(H)|-|N_{H}(V_1(H))|+(|C|-|N_{F}(C)|)>|V_1(H)|-|V_0(H)|.$$
Then defining $\tilde\varphi:V(H)\rightarrow [0, 1]$ as
$$
\tilde\varphi(v):=\left\{
\begin{array}{ccc}
1  & \text{ if }  v \in D, \\
0  & \text{ if }   v \in N_H(D), \\
\frac{1}{2}  & \text{otherwise},
\end{array}
\right.
$$
a contradiction can be obtained as before.
\end{proof}

\section{Universal Normal Approximation For Uniform Coloring}
\label{sec:normal}

Theorem \ref{th:birthdayuniformuniversal} says that if $\frac{1}{c}|E(G_n)|\stackrel{\sP}{\rightarrow} \infty$, then $N(G_n)$ converges to infinity as well. Since a Poisson random variable with mean growing to $\infty$ converges to a standard normal distribution after standardizing (centering by mean and scaling by standard deviation),  one possible question of interest is whether $N(G_n)$ properly standardized converges to a standard normal distribution. 
It turns out, as in the Poisson limit theorem, the normality of the standardized random variable $N(G_n)$ is universal 
whenever both $c=c(n)$ and $|E(G_n)|/c$ go to infinity. This will be proved by a similar method of moments argument. For this proof, without loss of generality we can assume that $E(G_n)\ge c$ almost surely, for every $n$. This is because $\P(E(G_n)/c< 1)$ converges to $0$ as $n\rightarrow\infty$, and so replacing the law of $G_n$ by the conditional law of $(G_n|G_n\ge c)$ does not affect any of the limiting distribution results. 


\subsection{Proof of Theorem \ref{th:normal}}Let $G_n\in \sG_n$ be a random graph sampled according to some probability distribution. This section proves a universal normal limit theorem for
\begin{equation}
Z_n:=\left(\frac{|E(G_n)|}{c}\right)^{-\frac{1}{2}}\sum_{(i,j)\in E(G_n)}\left\{\pmb 1\{Y_i=Y_j\}-\frac{1}{c}\right\}=\left(\frac{|E(G_n)|}{c}\right)^{-\frac{1}{2}}\left(N(G_n)-\frac{E(G_n)}{c}\right).
\label{eq:z}
\end{equation}
Associated with every edge of $G_n$ define the collection of random variables
$\{X_{ij}, (i, j)\in E(G_n)\}$, where $X_{ij} \text{ are i.i.d. } \dBer(1/c)$, and set 
\begin{equation}
W_n:=\left(\frac{|E(G_n)|}{c}\right)^{-\frac{1}{2}}\sum_{(i,j)\in E(G_n)}\left\{X_{(i, j)}-\frac{1}{c}\right\}=\left(\frac{|E(G_n)|}{c}\right)^{-\frac{1}{2}}\left(M(G_n)-\frac{E(G_n)}{c}\right).
\label{eq:w}
\end{equation}

\subsubsection{Comparing Conditional Moments} A multi-graph $G=(V, E)$ is a graph where multiple edges are allowed but there are no self loops. For a multi-graph $G$ denote by $G_S$ the simple graph obtained from $G$ by replacing the multiple edges with a single edge. A multi-graph $H$ is said to be a multi-subgraph of $G$ if the simple graph $H_S$ is a subgraph of $G$.

\begin{obs}Let $H=(V(H), E(H))$ be a multigraph with no isolated vertex. Let $F$ be a multigraph obtained by removing one edge from $H$ and removing all isolated vertices formed.
Then $|V(F)|-\nu(F)\geq |V(H)|-\nu(H)-1$.
\label{ob:graphinequality}
\end{obs}

\begin{proof}Observe that $\nu(F)\leq \nu(H)+1$ and $|V(H)|-2\leq |V(F)|\leq |V(H)|$. If $|V(F)|=|V(H)|$ the result is immediate. 

Now, if $|V(F)|=|V(H)|-1$, then the vertex removed must have degree 1 and so $\nu(F)=\nu(H)$, and the inequality still holds. 

Finally, if $\nu(F)=\nu(H)-2$, the edge removed must be an isolated edge, in which case the number of vertices decreases by two and the number of connected components decreases by one and the result holds.
\end{proof}

The above observation helps to determine the leading order of the expected central moments for multi-subgraphs of $G_n$.

\begin{lem}
For any multi-subgraph $H=(V(H), E(H))$ of $G_n$ define $$Z(H)=\prod_{(i, j)\in E(H)}\left\{\pmb 1\{Y_i=Y_j\}-\frac{1}{c}\right\}, \quad \text{ and } \quad W(H)=\prod_{(i, j)\in E(H)}\left\{X_{(i, j)}-\frac{1}{c}\right\}.$$ Then $\E(Z(H))\lesssim_H \frac{1}{c^{|V(H)|-\nu(H)}}$ and $\E(W(H))\lesssim_H \frac{1}{c^{|E(H_S)|}}$.
\label{lm:leadingorder}
\end{lem}

\begin{proof}By expanding the product,
\begin{equation}
Z(H)=\sum_{b=0}^{|E(H)|}\frac{(-1)^b}{c^b}\sum_{\substack{(i_s, j_s)\in E(H),\\ s\in [|E(H)|-b]}}\prod_{s=1}^{|E(H)|-b}\pmb 1\{Y_{i_s}=Y_{j_s}\},
\label{eq:wexpand}
\end{equation}
where the second sum is over all possible choices of $|E(H)|-b$ distinct multi-edges $(i_1,j_1), (i_2,j_2)\ldots $ $(i_{|E(H)|-b},j_{|E(H)|-b})$ from the multiset $E(H)$.

Let $F$ be the subgraph of $H$ formed by $(i_1,j_1), (i_2,j_2)\ldots (i_{|E(H)|-b},j_{|E(H)|-b})$.  Then by Observation \ref{ob:graphinequality}, $|V(F)|-\nu(F)\geq |V(H)|-\nu(H)-b$, and
\begin{equation}\frac{1}{c^{b}}\E\left(\prod_{s=1}^{|E(H)|-b}\pmb 1\{Y_{i_s}=Y_{j_s}\}\right)=\frac{1}{c^{|V(F)|-\nu(F)+b}}\leq \frac{1}{c^{|V(H)|-\nu(H)}}.
\label{eq:wterm}
\end{equation}
As the number of terms in (\ref{eq:wexpand}) depends only on $H$, and for every term (\ref{eq:wterm}) holds, the result follows.

The result for $W(H)$ follows similarly. The leading order of the expectation comes from the first term
$$\E\left(\prod_{(i, j)\in E(H)}X_{(i, j)}\right)=\frac{1}{c^{|E(H_S)|}},$$ and the number of terms depends only on $H$.
\end{proof}

Expanding the product also shows that the expected central moments of $Z(H)$ and $W(H)$ are equal when the underlying simple graph is a tree.

\begin{lem}
For any multi-subgraph $H=(V(H), E(H))$ of $G_n$, such that the underlying simple graph $H_S$ is a tree, $\E(Z(H))=\E(W(H))$.
\label{lm:tree_equal}
\end{lem}

\begin{proof}By (\ref{eq:wexpand})
\begin{eqnarray}
\E(Z(H))&=&\sum_{b=0}^{|E(H)|}\frac{(-1)^b}{c^b}\sum_{\substack{(i_s, j_s)\in E(H),\\ s\in [|E(H)|-b]}}\E\left(\prod_{s=1}^{|E(H)|-b}\pmb 1\{Y_{i_s}=Y_{j_s}\}\right)\nonumber\\
&=&\sum_{b=0}^{|E(H)|}\frac{(-1)^b}{c^b}\sum_{\substack{(i_s, j_s)\in E(H),\\ s\in [|E(H)|-b]}}\frac{1}{c^{|V(F)|-\nu(F)}}\nonumber\\
&=&\sum_{b=0}^{|E(H)|}\frac{(-1)^b}{c^b}\sum_{\substack{(i_s, j_s)\in E(H),\\ s\in [|E(H)|-b]}}\frac{1}{c^{|E(F_S)|}},
\label{eq:tree_equal}
\end{eqnarray}
where $F$ is the subgraph of $H$ formed by $(i_1,j_1), (i_2,j_2)\ldots (i_{|E(H)|-b},j_{|E(H)|-b})$. The last equality 
uses $|V(F)|-\nu(F)=|E(F_S)|$, because $F_S$ is a tree, since $H_S$ is a tree. The result now follows because every term in (\ref{eq:tree_equal}) is equal to every term in the expansion of $\E(W(H))$.
\end{proof}

Using the above lemmas and results about the fractional independence number $\gamma(H)$ from the previous section, the conditional moments of $Z_n$ and $W_n$ can be compared. For a simple graph $G$ and a multigraph $H$ define
$$M(G, H)=\sum_{e_1\in E(G)}\sum_{e_2\in E(G)}\cdots \sum_{e_{|E(H)|}\in E(G)}\pmb 1\{G[e_1, e_2, \ldots e_{|E(H)|}]=H\},$$
where $G[e_1, e_2, \ldots e_{|E(H)|}]$ is the multi-subgraph of $G$ formed by the edges $e_1, e_2, \ldots, e_{|E(H)|}$. It is easy to see that $M(G, H)\lesssim_H N(G, H_S)$.

\begin{lem}
Let $Z_n$ and $W_n$ be as defined in (\ref{eq:z}) and (\ref{eq:w}) respectively, and $\cM_k$ be the set of all multi-graphs of $G_n$ with exactly $k$ multi edges and $d_{\min}(H)\ge 2$ and $\gamma(H)=|E(H)|/2$. Then for every $k\geq 1$,
$$\left||\E (Z_n^k|G_n)-\E (W_n^k|G_n)|-\Big|\left(\frac{|E(G_n)|}{c}\right)^{-\frac{k}{2}}\sum_{\substack{H\in \cM_k\\ }}M(G_n, H)\E(W(H))-\E(Z(H))\Big|\right|\pto0,$$
whenever $|E(G_n)|/c\pto \infty$, irrespective of whether $c\rightarrow\infty$ or $c$ is fixed.
\label{lm:moment_comparison}
\end{lem}

\begin{proof}
By the multinomial expansion,
\begin{eqnarray}
\E(Z_n^k|G_n)-\E(W_n^k|G_n)&=&
\left(\frac{|E(G_n)|}{c}\right)^{-\frac{k}{2}}\sum_{\substack{H\in \cH_k\\ }}M(G_n, H)\left(\E(Z(H))-\E(W(H))\right),
\label{eq:zn}
\end{eqnarray}
where $\cH_k$ is the set of all multi-subgraphs of $G_n$ with exactly $k$ multiple edges and no isolated vertex.

Let $\cN_k$ be the collection of all multi-subgraphs $H$ of $G_n$ with $d_{\min}(H)=1$. To begin with assume that $H\in \cN_k$ and that the labelled vertex 1 has degree 1. Suppose, vertex $s \in [n]\backslash \{1\}$ is the only neighbor of 1 and $Y_{-1}=\{Y_1, Y_2, \ldots, Y_n\}\backslash \{Y_1\}$. Therefore,
$$\E(Z(H)|Y_{-1}, G_n)=\left(\E(\pmb 1\{Y_1=Y_s\}|Y_{-1}, G_n)-\frac{1}{c}\right)\prod_{\substack{(i, j)\in E(H),\\ (i, j)\ne (1, s)}}\left\{\pmb 1\{Y_i=Y_j\}-\frac{1}{c}\right\}=0,$$
which implies $\E(Z(H)|G_n)=0$. Similarly, $\E(W(H)|G_n)=0$. Therefore, for all graphs $H\in \cN_k$, the corresponding term in the sum (\ref{eq:zn}) is zero. 

Now, for $H\in \cH_k$ such that $\gamma(H)<|E(H)|/2$,
\begin{eqnarray}
\left(\frac{|E(G_n)|}{c}\right)^{-\frac{|E(H)|}{2}}M(G_n, H)\cdot\E(Z(H))&\lesssim_H&\left(\frac{|E(G_n)|}{c}\right)^{-\frac{|E(H)|}{2}} \frac{N(G_n, H_S)}{c^{|V(H)|-\nu(H)}}\nonumber\\
&\lesssim_H& \frac{|E(G_n)|^{\gamma(H)-\frac{1}{2}|E(H)|}}{c^{|V(H)|-\nu(H)-\frac{1}{2}|E(H)|}}\nonumber\\
&\lesssim_H& \frac{\omega^{\gamma(H)-\frac{1}{2}|E(H)|}}{c^{|V(H)|-\nu(H)-\gamma(H)}}\stackrel{\sP}{\rightarrow} 0,
\label{eq:w_h}
\end{eqnarray}
whenever $\omega=|E(G_n)|/c\pto \infty$, since $\gamma(H)\leq |V(H)|-\nu(H)$ by Corollary \ref{cor:gammah}.

Similarly, for $H\in \cH_k$ such that $\gamma(H)<|E(H)|/2$,
\begin{eqnarray}
\left(\frac{|E(G_n)|}{c}\right)^{-\frac{|E(H)|}{2}}M(G_n, H)\cdot\E(W(H))&\lesssim_H&\left(\frac{|E(G_n)|}{c}\right)^{-\frac{|E(H)|}{2}} \frac{N(G_n, H_S)}{c^{|E(H_S)|}}\nonumber\\
&\lesssim_H&\left(\frac{|E(G_n)|}{c}\right)^{-\frac{|E(H)|}{2}} \frac{N(G_n, H_S)}{c^{|V(H)|-\nu(H)}}\stackrel{\sP}{\rightarrow} 0,
\label{eq:u_h}
\end{eqnarray}
whenever $\omega=|E(G_n)|/c\pto \infty$.

Combining (\ref{eq:w_h}) and (\ref{eq:u_h}) and using Lemma \ref{lm:degree_one} the result follows.
\end{proof}

Note that the above lemmas are true irrespective of whether $c\rightarrow \infty$ or $c$ is fixed. Hence, they will be relevant even when we are dealing with the case $c$ is fixed. In the following lemma, it is shown that the remaining terms in (\ref{eq:zn}) are also negligible when $c\rightarrow \infty$.

\begin{lem}Let $Z_n$ and $W_n$ be as defined in (\ref{eq:z}) and (\ref{eq:w}) respectively. If $c\rightarrow \infty$ and $|E(G_n)|/c\stackrel{\sP}{\rightarrow} \infty$, then for every fixed $k\geq 1$ 
$$|\E (Z_n^k|G_n)-\E (W_n^k|G_n)|\stackrel{\sP}{\rightarrow}0.$$
\label{lm:normal_difference_zero}
\end{lem}

\begin{proof}Let $\cM_k$ be as defined in Lemma \ref{lm:moment_comparison}, and $\sS_k\subset \cM_k$ be the set of all multi-graphs $H$ with $d_{\min}(H)\ge 2, |E(H)|=k \text{ and } \gamma (H)=|E(H)|/2=|V(H)|-\nu(H)$. By Lemma \ref{lm:moment_comparison} and using $c\rightarrow \infty$ in (\ref{eq:w_h}) and (\ref{eq:u_h}) it follows that only the multi-subgraphs of $G_n$ which are in $\sS_k$ need to considered. 
Now, for $H\in \sS_k$,
$$\gamma(H_S)=\gamma (H)=|V(H)|-\nu(H)=|V(H_S)|-\nu(H_S).$$ Therefore, $H_S$ is a disjoint union of stars by Corollary \ref{cor:gammah}. 
Moreover, $$|E(H_S)|=|V(H_S)|-\nu(H_S)=|V(H)|-\nu(H)=|E(H)|/2.$$ This, along with the fact that $H$ cannot have any vertex of degree 1 
gives that any $H\in \sS_k$ is a disjoint union of stars, where every edge is repeated twice. Now, for any such graph $H$, $\E(Z(H))=\E(W(H))$ by Lemma \ref{lm:tree_equal}, and the result follows.
\end{proof}

\subsubsection{Completing the Proof of Theorem \ref{th:normal}}

To complete the proof, it remains to shows that $W_n$ satisfies the conditions of Lemma \ref{abc}. This is verified in the following lemma:

\begin{lem}\label{obvious2}
Let $W_n$ be as defined in (\ref{eq:w}). Then $W_n\stackrel{\sD}{\rightarrow}N(0,1)$, and further for any $\varepsilon>0$ and $t$ small enough,
\begin{align}
\limsup_{k\rightarrow\infty}\limsup_{n\rightarrow \infty}\P\left(\left|\frac{t^k}{k!}\E(W_n^k|G_n)\right|>\varepsilon\right)=0.
\label{eq:obvious_moment}
\end{align}
\end{lem}

\begin{proof}
To prove the first conclusion, let
\begin{equation*}
\overline{W}_n:=\frac{M(G_n)-\frac{|E(G_n)|}{c}}{\sqrt{\frac{|E(G_n)|}{c}-\frac{|E(G_n)|}{c^2}}}.
\end{equation*}
By the Berry-Esseen theorem $$\left|\P(\overline{W}_n\le x|G_n)-\Phi(x)\right|\le K\sqrt{\frac{c}{|E(G_n)|}},$$ where $\Phi(x)$ is the standard normal density and $K<\infty$ is some universal constant. Now, the right hand side converges to $0$ in probability, and so $\overline{W}_n$ converges to $N(0,1)$ by the Dominated Convergence theorem. Thus,  $$\overline{W}_n-W_n=\overline{W}_n\left(1-\frac{1}{c}\right)^{1/2}\pto0$$ by Slutsky's theorem, which implies $W_n$ converges in distribution to $N(0, 1)$. 

For the second conclusion, it suffices to show that
\begin{equation}
\sup_{n\geq 1}\E\left(e^{tW_n}\right)<\infty
\label{eq:exponential_moment}
\end{equation}
for $|t|<\delta$ small enough. This is because (\ref{eq:exponential_moment}) implies $\{W_n^{2m}\}_{n \geq 1}$ is uniformly integrable and $\lim_{n\rightarrow\infty}\E(W_n^{2m})=\frac{(2m)!}{2^m m!}$ for all $m\geq 1$. Therefore, by Markov's inequality,
\begin{eqnarray*}\limsup_{k\rightarrow \infty}\limsup_{n \rightarrow\infty}\P\left(\left|\frac{t^k}{k!}\E(W_n^k|G_n)\right|>\varepsilon\right)&\leq & \limsup_{k\rightarrow \infty}\limsup_{n \rightarrow\infty}\frac{1}{\varepsilon}\E\left(\frac{t^k}{k!}|W_n|^k\right)\nonumber\\
&\leq & \limsup_{k\rightarrow \infty}\limsup_{n \rightarrow\infty}\frac{1}{\varepsilon}\frac{t^k}{k!}\E\left(|W_n|^{2k}\right)^{\frac{1}{2}}=0,
\end{eqnarray*}
which verifies (\ref{eq:abc2}).

It remains to verify (\ref{eq:exponential_moment}). Define $\sigma_n^2:=|E(G_n)|/c$, and so $$\E\left(e^{tW_n}|G_n\right)=e^{-\frac{t|E(G_n)|}{\sigma_n c}}\left(1-\frac{1}{c}+\frac{e^{t}}{c}\right)^{|E(G_n)|}:=R_n.$$ Therefore, for any $t\in \R$ we have
\begin{eqnarray*}
\log R_n&=&-\frac{t|E(G_n)|}{\sigma_n c}+|E(G_n)|\log\left(1-\frac{1}{c}+\frac{e^{\frac{t}{\sigma_n}}}{c}\right)\nonumber\\
&\leq &-\frac{t|E(G_n)|}{\sigma_n c}+|E(G_n)|\left(\frac{e^{\frac{t}{\sigma_n}}-1}{c}\right)\nonumber\\
&\le&\frac{|E(G_n)|}{c}\left(\frac{t^2}{2\sigma_n^2}+e^t\frac{|t|^3}{6\sigma_n^3}\right)\le\frac{t^2}{2}+\frac{e^t|t|^3}{6},
\end{eqnarray*}
 where the last two  inequalities follows from Taylor's series and the fact that $|E(G_n)|\ge c$. This implies that  $\E(R_n)< \infty$, as desired.

\end{proof}

%

\section{Normal Limit Theorem for Fixed Number of Colors}
\label{sec:normal_fixed}

The condition $c:=c(n)\rightarrow\infty$ in Theorem \ref{th:normal} is necessary for the universal normality.
The following example demonstrates that $N(G_n)$ is not asymptotically normal, when the number of colors $c$ remains fixed.

\begin{example}
Consider coloring the graph $G_n=K_{2, n}$ with $c=2$ colors.  It is easy to see that
$N(G_n)$ is $2nU_n$ or $n$ with probability $\frac{1}{2}$ each, where $nU_n\sim \dBin(n, 1/2)$. This implies that  $$n^{-\frac{1}{2}}\left(N(G_n)-n\right)\stackrel{\sD}{\rightarrow}\frac{1}{2}N(0,1)+\frac{1}{2}\delta_0,$$
a mixture of a standard normal and point mass at 0.
\label{ex:K2n}
\end{example}

Note that in the previous example the ACF4 condition is not satisfied: $N(G_n, C_4)={n \choose 2}$, and $\lim_{n \rightarrow \infty}\frac{N(G_n, C_4)}{|E(G_n)|^2}=1/8$.

\subsection{The ACF4 Condition}
\label{sec:4cycle}

Recall the ACF4 condition $N(G_n, C_4)=o_P(|E(G_n)|^{2})$, that is, the number of copies of the 4-cycle $C_4$ in $G_n$ is sub-extremal. The following theorem shows this implies that the number of copies of the $g$-cycle $C_g$ in $G_n$ is also sub-extremal, for all $g \geq 3$. The theorem works for both deterministic and random sequence of graphs. For the sake of clarity we write the proof for deterministic graphs, noting that the exact proof goes through if the graphs are random.

\begin{thm}\label{th:4cycle}
The ACF4 condition $N(G_n, C_4)=o_P(|E(G_n)|)^{2}$ is equivalent to the condition that $N(G_n, C_g)=o_P(|E(G_n)|)^{g/2}$, for all $g \geq 3$. 
\end{thm}

\begin{proof} 
Let $A(G_n)=((a_{ij}))$ be the adjacency matrix of the graph $G_n$. For any two vertices $a, b\in V(G_n)$, let $s_2(a, b)$ be the number of common neighbors of $a, b$. It is easy to see that 
\begin{equation}
o(|E(G_n)|^2)=N(C_4, G_n)\asymp \sum_{\substack{a, b\in V(G_n)\\ s_2(a, b)\geq 2}}{s_2(a, b)\choose 2}\gtrsim \sum_{\substack{a, b\in V(G_n)\\ s_2(a, b)\geq 2}}s_2(a, b)^2.
\label{eq:2star1}
\end{equation}
Moreover, $\sum_{a, b \in V(G_n)}s_2(a, b)\asymp N(G_n, K_{1, 2})= O(|E(G_n)|^2)$. Finally, observe that for any $m \geq 2$
\begin{equation}
\sum_{\substack{a, b\in V(G_n)\\ s_2(a, b)\geq m}}{s_2(a, b)}\leq \sum_{\substack{a, b\in V(G_n)\\ s_2(a, b)\geq m}}s_2(a, b)^2/m= o(|E(G_n)|^2)/m,
\label{eq:2star2}
\end{equation}
where the last step uses (\ref{eq:2star1}).

Now fix $\varepsilon>0$ and consider the following two cases depending on whether the length of the cycle is even or odd.

\begin{description}
\item[1] Suppose $g=2h+1\geq 3$. It is easy to see that
\begin{equation}
N(G_n, C_g)\lesssim \sum_{i_2, i_3, \ldots, i_g\in V(G_n)}s_2(i_2, i_g)\prod_{j=2}^{g-1}a_{i_{j}i_{j+1}}.
\label{eq:oddcycle}
\end{equation}
Note that for a path $P_{2b+1}$ with $2b+1$ edges  we have $N(G_n, P_{2b+1})=O(|E(G_n)|^{b+1}$, as $\gamma(P_{2b+1})=b+1$ (see Alon \cite[Corollary 1]{alon81}). Then 
\begin{equation}
\sum_{\substack{i_2, i_3, \ldots, i_g\in V(G_n),\\ s_2(i_2, i_g)\leq \varepsilon(|E(G_n)|)^{1/2}}}s_2(i_2, i_g)\prod_{j=2}^{g-1}a_{i_{j}i_{j+1}}\lesssim \varepsilon(|E(G_n)|)^{1/2}N(G_n, P_{2h-1})\lesssim_g \varepsilon |E(G_n)|^{g/2}.
\label{eq:oddcycle1}
\end{equation}
Also, using (\ref{eq:2star2})
\begin{eqnarray}
\sum_{\substack{i_2, i_3, \ldots, i_g\in V(G_n),\\ s_2(i_2, i_g)> \varepsilon(|E(G_n)|)^{1/2}}}s_2(i_2, i_g)\prod_{j=2}^{g-1}a_{i_{j}i_{j+1}}&\leq& \sum_{\substack{i_2, i_g\\s_2(i_2, i_g)\geq \varepsilon |E(G_n)|^{1/2}}}s_2(i_2, i_g)\sum_{i_3, i_4, \ldots, i_{g-2}\in V(G_n)}\prod_{j=3}^{g-2}a_{i_{j}i_{j+1}}\nonumber\\
&\lesssim_g&\varepsilon^{-1}o((|E(G_n)|)^{3/2})N(G_n, P_{2h-3})\nonumber\\
&\lesssim_g& \varepsilon^{-1}o(|E(G_n)|^{g/2}),
\label{eq:oddcycle2}
\end{eqnarray}
where $N(G_n, P_{-1}):=1$ by definition. Combining (\ref{eq:oddcycle1}) and (\ref{eq:oddcycle2}) with (\ref{eq:oddcycle}) it follows that $\lim_{n\rightarrow \infty} N(G_n, C_g)/ |E(G_n)|^{g/2}\lesssim_g\varepsilon$. Since $\varepsilon$ is arbitrary, this implies that $N(G_n, C_g)=o(|E(G_n)|^{g/2})$ for $g$ odd.

\item[2] Suppose $g=2h\geq 6$. The result will be proved by induction on $h$. The base case $h=2$ is true by assumption. Now, suppose $h\geq 2$ and $N(G_n, C_{2h})=o(|E(G_n)|^h)$. For vertices $a, b\in V(G)$, let $s_h(a, b)$ be the number of paths with $h$ edges in $G_n$ with one end point at $a$ and another at $b$. Therefore, as in (\ref{eq:2star1})
\begin{equation}
\sum_{\substack{a, b\in V(G_n)\\ s_h(a, b)\geq 2}}s_h(a, b)^2\lesssim_h N(G_n, C_{2h})=o(|E(G_n)|^h).
\label{eq:2path1}
\end{equation}
Moreover, $\sum_{a, b \in V(G_n)}s_h(a, b)\asymp N(G_n, P_h)= O(|E(G_n)|^h)$. Finally, as in (\ref{eq:2star2}), for any $m \geq 2$
\begin{equation}
\sum_{\substack{a, b\in V(G_n)\\ s_h(a, b)\geq m}}{s_h(a, b)}\leq \sum_{\substack{a, b\in V(G_n)\\ s_h(a, b)\geq m}}s_h(a, b)^2/m= o(|E(G_n)|^h)/m,
\label{eq:2path2}
\end{equation}
where the last step uses (\ref{eq:2path1}).

Now, it is easy to see that
\begin{equation}
N(G_n, C_{2h+2})\lesssim_h \sum_{x, y, u, v\in V(G_n)}s_h(x, y)s_h(u, v)a_{xu}a_{yv}.
\label{eq:evencycle}
\end{equation}
Let $S=\{x, y, u, v\in V(G_n): s_h(x, y)s_h(u, v)\leq \varepsilon|E(G_n)|^{h-1}\}$. Note that
\begin{equation}
\sum_{S}s_h(x, y)s_h(u, v)a_{xu}a_{yv}\leq \varepsilon|E(G_n)|^{h-1} \sum_{x, y, u, v\in V(G_n)}a_{xu}a_{yv}\lesssim \varepsilon|E(G_n)|^{h+1}.
\label{eq:6cycle1}
\end{equation}
Also, using (\ref{eq:2star2})
\begin{eqnarray}
\sum_{\overline S}s_h(x, y)s_h(u, v)a_{xu}a_{yv}&\leq& \sum_{\overline S}\sum_{r} r \pmb 1\{r=s_h(x, y)\} s_h(u, v)\pmb 1_{\left\{s_h(u, v)>\frac{\varepsilon |E(G_n)|^{h-1}}{r}\right\}} \nonumber\\
&\leq &\sum_{x, y}\sum_{r}r \pmb 1\{r=s_h(x, y)\}\sum_{u, v} s_h(u, v)\pmb 1_{\left\{s_h(u, v)>\frac{\varepsilon |E(G_n)|^{h-1}}{r}\right\}}\nonumber\\
&\leq & \sum_{x, y}\sum_{r} r \pmb 1\{r=s_h(x, y)\} \left(\frac{o(|E(G_n)|^h)}{\varepsilon |E(G_n)|^{h-1}/r}\right)\nonumber\\
&\leq & \left(\frac{o(|E(G_n)|)}{\varepsilon}\right) \sum_{x, y}\sum_{r} r^2 \pmb 1\{r=s_h(x, y)\}\nonumber\\
&\lesssim & \left(\frac{o(|E(G_n)|)}{\varepsilon}\right) \sum_{x, y}s_h(x, y)^2=\frac{o(|E(G_n)|^{h+1})}{\varepsilon}.
\label{eq:6cycle2}
\end{eqnarray}
Combining (\ref{eq:6cycle1}) and (\ref{eq:6cycle2}) with (\ref{eq:evencycle}) it follows that $\lim_{n\rightarrow \infty} N(G_n, C_{2h+2})/ |E(G_n)|^{h+1}\lesssim_h \varepsilon$. Since $\varepsilon$ is arbitrary, this implies that $N(G_n, C_{2h+2})=o(|E(G_n)|^{h+1})$. This completes the induction step, and hence completes the proof.
\end{description}
\end{proof}

\begin{remark}The above theorem shows that if the number of 4-cycles in a graph is sub-extremal, then the number of copies of any cycle graph is also sub-extremal. This is an illustration of the fourth moment phenomenon: the convergence of all moments of $N(G_n)$ is implied solely by the convergence of the  fourth moment of $N(G_n)$. In extremal combinatorics of pseudo-random graphs, 4-cycles play a similar role: the classic result of Chung et al. \cite{graham} asserts that if the edge and 4-cycle densities of a graph are as in a binomial random graph, then the graph is essentially pseudo-random and the density of any other subgraph is  like that in a binomial random graph. Recently, Conlon et al. \cite{conlonfoxzhao} proved similar results in the sparse regime.
\end{remark}

Cycle counts in graphs are closely related to the sum of powers of eigenvalues of the adjacency matrix. If $\vec\lambda(G_n)=(\lambda_1(G_n), \lambda_2(G_n),\cdots, \lambda_n(G_n))'$ is the vector eigenvalues of the adjacency matrix $A(G_n)$, then  $\sum_{i=1}^n\lambda_i^g(G_n)$ counts the number of closed walks of length $g$ in the graph $G_n$. Analogous to the ACF4 condition, a sequence of random graphs $\{G_n\}_{n\geq 1}$ is said to satisfy the {\it uniform spectral negligibility (USN) condition} if 
\begin{equation}
\lim_{n\rightarrow\infty}\frac{\max_{i \in [n]}|\lambda_i(G_n)|}{\left(\sum_{j=1}^n\lambda_j^2(G_n)\right)^{\frac{1}{2}}}=0,
\label{eq:spectral_condition}
\end{equation}
If $\pmb \lambda(G_n)=\frac{\vec\lambda(G_n)}{||\vec\lambda(G_n)||_2}$, is the vector of normalized eigenvalues, then USN condition can be rewritten as $\lim_{n\rightarrow\infty}||\pmb \lambda(G_n)||_\infty=0$.

\begin{obs}\label{obs:usn}
If a sequence of graphs satisfies the USN condition, then it also satisfies the ACF4 condition. 
\end{obs}

\begin{proof}If a sequence of graphs satisfies the USN condition, then for every $g\geq 3$
$$N(G_n, C_g)\lesssim_g \sum_{i=1}^n\lambda_i^g(G_n)\leq \left(\sum_{i=1}^n\lambda_i^2(G_n)\right) ||\vec \lambda(G_n)||_\infty^{g-2}=o(||\vec \lambda(G_n)||_2^g)=o(|E(G_n)|^{g/2}),$$
as $||\vec \lambda(G_n)||_2^2=2|E(G_n)|$.
\end{proof}

More details about the differences between the ACF4 and USN conditions are presented in Section \ref{sec:star}.

\subsection{Proof of Theorem \ref{th:normal_fixed}}

We begin with the following lemma which shows that the ACF4 condition ensures that the counts of any graph with a cycle is also sub-extremal. Recall that if $\varphi$ is an optimal solution of (\ref{eq:gamma}), then $V(H)=V_0(H)\cup V_{1/2}(H)\cup V_1(H)$, where $V_a(H)=\{v\in V(H): \varphi(v)=a\}$, for $a\in \{0, 1/2, 1\}$.

\begin{lem}Let $H$ be a multi-subgraph of $G_n$ with no isolated vertex, $d_{\min}(H)\geq 2$, $\gamma(H)=|E(H)|/2$, and at least 3 vertices in one of its connected components. If the ACF4 condition holds for $G_n$, then
$$M(G_n, H)=o_P(|E(G_n)|^{|E(H)|/2}),$$
whenever the girth $g(H)\geq 3$.
\label{lm:count_cycle}
\end{lem}

\begin{proof} As in the proof of Theorem \ref{th:4cycle} we give the proof for deterministic graphs only. Further, note that it suffices to prove that $N(G_n, H)=o(|E(G_n)|^{|E(H)|/2})$ for any simple subgraph $H$ of $G_n$. This is because for any multi-subgraph $H$ of $G_n$, $M(G_n, H)\lesssim_H N(G_n, H_S)= o(|E(G_n)|^{|E(H_S)|/2})=o(|E(G_n)|^{|E(H)|/2})$, since $|E(H)|\geq |E(H_S)|$.

To begin with assume that $H$ is connected. If $|E(H)|=|V(H)|-1$, the graph $H$ is a tree and $d_{\min}(H)=1$. Also, if $|E(H)|=|V(H)|$, then since $|V(H)|\geq 3$ and $d_{\min}(H)\geq 2$, the only possibility is that $H=C_g$ for some $g\geq 3$, and $N(G_n, H)=o(|E(G_n)|^{|E(H)|/2})$ by assumption.

Therefore, it suffices to consider a subgraph $H$ of $G_n$ such that $|E(H)|>|V(H)|$. If $\gamma(H)\leq |V(H)|/2$, $$N(G_n, H)\lesssim_H |E(G_n)|^{\gamma(H)}\lesssim_H|E(G_n)|^{|V(H)|/2}=o(|E(G_n)|^{|E(H)|/2}).$$ Therefore, assume that $\gamma(H)> |V(H)|/2$. As in the proof of Lemma \ref{lm:gamma_structural}, let $\varphi:V(H)\rightarrow [0, 1]$ be an extreme point of $\sP(H)$ that is the optimal solution linear program defined in (\ref{eq:gamma}). Partition $V(H)=V_0(H)\cup V_{1/2}(H)\cup V_1(H)$, where $V_a(H)=\{v\in V(H): \varphi(v)=a\}$, for $a\in \{0, 1/2, 1\}$. Note that $\gamma(H)> |V(H)|/2$ implies that $\varphi$ is not identically equal to 1/2. Depending upon the size of $V_{1/2}(H)$ the following cases arise:

\begin{description}
\item[$|V_{1/2}(H)|\ne 0$] Decompose $H$ into subgraphs $H_{01}$ and $H_{1/2}$ by Lemma \ref{lm:gamma_structural}. $H_{01}=(V_0(H)\cup V_1(H), E(H_{01}))$ is a bipartite graph where $E(H_{01})$ is set of edges from $V_0(H)$ to $V_1(H)$, has a matching which saturates every vertex in $V_0(H)$. Therefore,
\begin{equation}
N(G_n, H_{01})\lesssim_H |E(G_n)|^{|V_1(H)|}\lesssim_H |E(G_n)|^{|E(H_{01})|/2},
\label{eq:H01}
\end{equation}
since $d_{\min}(H)\geq 2$ implies $|E(H_{01})|\geq 2|V_1(H)|$.
Moreover, the subgraph $F$ of $H$ induced by the vertices of $V_{1/2}(H)$ has a spanning subgraph which is a disjoint union of cycles and isolated edges. Therefore, by Theorem 4 of Alon \cite{alon81}, $\gamma(F)=|V_{1/2}(H)|/2$.
\begin{equation}
N(G_n, F)\lesssim_H |E(G_n)|^{|V_{1/2}(H)|/2}.
\label{eq:H1/2}
\end{equation}
Now, let $F_1, F_2, \ldots, F_{\nu}$ be the connected components of $F$. Denote by $E(V(F), V_0(H))$ the subset of edges in $H$ with one vertex in $V(F_{i})$ and another in $V_0(H)$, for $i\in [\nu]$. Therefore, using estimates (\ref{eq:H01}) and (\ref{eq:H1/2}),
\begin{eqnarray}
\frac{N(G_n, H)}{|E(G_n)|^{|E(H)|/2}}&\lesssim_H & \frac{N(G_n, H_{01})\prod_{i=1}^{\nu}N(G_n, F_i)}{|E(G_n)|^{|E(H)|/2}}\nonumber\\
&\lesssim_H &\prod_{i=1}^{\nu}\frac{|E(G_n)|^{|V(F_i)|/2-|E(F_i)|/2}}{|E(G_n)|^{|E(V(F_{i}), V_{0}(H))|/2}}\nonumber\\
&=&\prod_{i=1}^{\nu}|E(G_n)|^{\lambda(F_i)/2},
\end{eqnarray}
where $\lambda(F_i):=|V(F_i)|-|E(F_i)|-|E(V(F_{i}), V_{0}(H))|$. Note that $|E(V(F_{i}), V_{0}(H))|>0$, since $H$ is connected. Hence, $|E(G_n)|^{\lambda(F_i)/2}=o(1)$ whenever $|V(F_i)|\leq |E(F_i)|$. Otherwise $F_i$ is a tree and has at least 2 vertices of degree 1. The degree 1 vertices must be connected to some vertex in $H_0$, which implies that $|E(V(F_{i}), V_{0}(H))|\geq 2$ and again $|E(G_n)|^{\lambda(F_i)/2}=o(1)$.

\item[$|V_{1/2}(H)|=0$] In this case, $V(H)=V_0(H)\cup V_1(H)$, that is every vertex is assigned the value 0 or 1 by the optimal function $\varphi$. By Lemma \ref{lm:degree_one}, for every vertex $v\in V_1(H)$, $d(v)=2$. Therefore, $|E(H)|=2|V_1(H)|=2\gamma(H)$ and the graph $H$ is bi-partite. By Lemma \ref{lm:gamma_structural} $H$ then has a matching which saturates every vertex in $V_0(H)$. By assumption, the girth $g:=g(H)\geq 3$ and let $F$ be a subgraph of $H$ which is isomorphic to $C_g$. Define
$H^-:=(V(H^-), E(H^-))$ where  $V(H^-)=V(H)\backslash V(F)$ and $E(H^-)=E(H)\backslash E(F)$.
Now, let $A\subset V_0(H^-):= V_0(H)\backslash F$. By the saturating matching in $H$, $|N_H(A)|\geq |A|$. Also, observe that $|N_{H^-}(A)|=|N_H(A)|$, since removing $F$ from $H$ leaves the vertices in $A$ and its neighbors unchanged. Therefore, $|N_{H^-}(A)|\geq |A|$ for all $A\subset V_0(H^-)$, that is, even after removing the cycle $F$ from $H$, there is a matching in $H^{-}$ which saturates every vertex in $V_0(H^-)$. This implies, $$N(G_n, H^-)\lesssim_H|E(G_n)|^{|V(H^-)|}.$$
Now, as $|V(H^-)|=|V_1(H)|-g/2=|E(H)|/2-g/2$ and $|E(H^-)|=|E(H)|-g$,
$$\frac{N(G_n, H)}{|E(G_n)|^{|E(H)|/2}}\lesssim_H\frac{N(G_n, F)N(G_n, H^-)}{|E(G_n)|^{|E(H)|/2}}\lesssim_H\frac{N(G_n, F)}{|E(G_n)|^{g/2}}=o(1).$$

\end{description}
Finally, if $H$ is not connected, let $H_1, H_2, \ldots, H_r$ be the connected components of $H$. There exists $j \in [r]$ such that $d_{\min}(H_j)\geq 2$ and $g(H_j)\geq 3$. By the applying the above argument for $H_j$, it follows that $N(G_n, H_j)=o(|E(G_n)|^{|E(H_j)|/2})$.
Also, for all $i \in [r]$, $N(G_n, H_i)\lesssim_H |E(G_n)|^{\gamma(H_i)}\leq |E(G_n)|^{|E(H_i)|/2}$, by Lemma \ref{lm:degree_one}. This implies,
$$N(G_n, H)\lesssim_H\prod_{i=1}^r N(G_n, H_i)=o(|E(G_n)|^{|E(H_j)|/2})O(|E(G_n)|^{(|E(H_i)|2})=o(|E(G_n)|^{|E(H)|/2}),$$
and the result follows.
\end{proof}

The above result combined with Lemma \ref{lm:moment_comparison} shows that the conditional moments of 
$Z_n$ and $W_n$, as defined in (\ref{eq:z}) and (\ref{eq:w}), are close whenever the ACF4 condition holds in probability.

\begin{lem}
Let $c$ be fixed and $\{G_n\}_{n\geq 1}$ be sequence of random graphs for which the ACF4 condition in probability holds. With $Z_n$ and $W_n$ as defined in (\ref{eq:z}) and (\ref{eq:w}), for every fixed $k\geq 1$
$$|\E (Z_n^k|G_n)-\E (W_n^k|G_n)|\stackrel{\sP}{\rightarrow}0.$$
\label{lm:normal_difference_zero_c_fixed}
\end{lem}

\begin{proof} From (\ref{eq:zn}),
$$\E(Z_n^k|G_n)-\E(W_n^k|G_n)=\left(\frac{|E(G_n)|}{c}\right)^{-\frac{k}{2}}\sum_{\substack{H\in \cH_k\\ }}M(G_n, H)\left(\E(Z(H))-\E(W(H))\right).$$
By Lemma \ref{lm:moment_comparison}, in the above sum it suffices to only consider graphs in $\cM_k$, that is, the set of all multi-graphs with exactly $k$ multi edges, no isolated vertex, $d_{\min}(H)\ge 2$, and  $\gamma (H)=|E(H)|/2$. 

If $H\in \cM_k$ is such that the girth $g(H)\geq 3$, then by Lemma \ref{lm:count_cycle} $M(G_n, H)=o_P(|E(G_n)|^{k/2})$. Therefore, the only multi-subgraphs $H\in \cM_k$ which remain must be such that $H_S$ is a tree.
But, by Lemma \ref{lm:tree_equal}, $\E(Z(H))=\E(W(H))$ for all such multi-subgraphs, and the result follows.
\end{proof}

\subsubsection{Completing the Proof of Theorem \ref{th:normal_fixed}}
%
Define $S_n:=(1-1/c)^{-\frac{1}{2}}W_n$, where $W_n$ is as defined before. Then $S_n\stackrel{\sD}{\rightarrow}N(0,1)$ by the standard central limit theorem. Moreover, by a direct application of Lemma \ref{obvious2}
\begin{align}
\limsup_{k\rightarrow\infty}\limsup_{n\rightarrow \infty}\P\left(\left|\frac{t^k}{k!}\E(W_n^k|G_n)\right|>\varepsilon\right)=0.
\end{align}
Therefore, $Z_n\stackrel{\sD}\rightarrow N(0, 1)$ by Lemma \ref{abc}.

To prove the necessity, assume that $\frac{N(G_n, C_4)}{|E(G_n)|^2}\stackrel{\sP}\nrightarrow 0$. Therefore, $\limsup_{n\rightarrow\infty}\E\left(\frac{N(G_n, C_4)}{|E(G_n)|^2}\right)> 0$.

Let $\sT_4$ be the collection of all multi-graphs with 4 edges and no vertex of degree 1. By Lemma \ref{lm:degree_one},
\begin{eqnarray}
\E(Z_n^4|G_n)=\left(\frac{|E(G_n)|}{c}\right)^{-2}\sum_{H \in \sT_4}M(G_n, H)\E(Z(H))
\label{lm:fourth_moment}
\end{eqnarray}

Now, it is easy to see that $\sT_4$ consists of 5 multi-graphs, which are shown in Figure \ref{fig:fourth_moment_graphs}.

\begin{figure*}[h]
\centering
\begin{minipage}[c]{1.0\textwidth}
\centering
\includegraphics[width=4.75in]
    {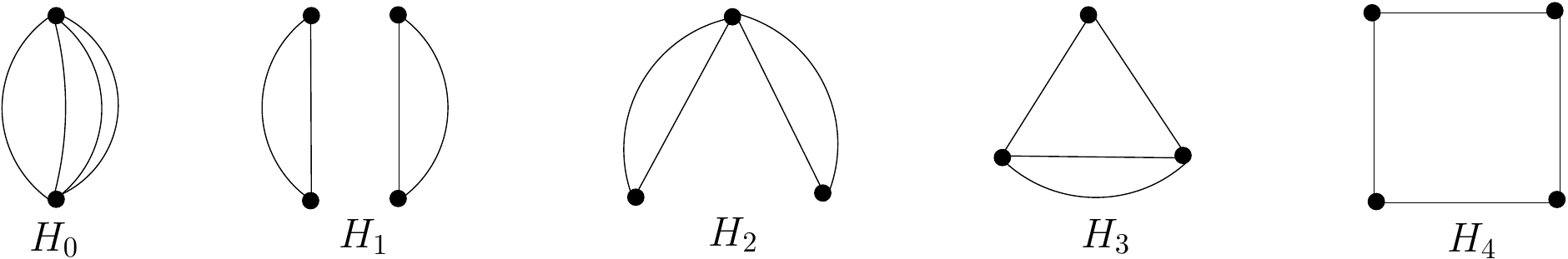}\\
\end{minipage}
\caption{The multi-graphs which arise in the fourth moment calculation}
\label{fig:fourth_moment_graphs}
\end{figure*}

Note that $N(G_n, H_0)\lesssim_H |E(G_n)|=o(|E(G_n)|^2)$ and $N(G_n, H_3)\lesssim_H |E(G_n)|^{\gamma(H_3)}=|E(G_n)|^{3/2}=o(|E(G_n)|^2)$. Also, by direct calculations
$$\E(Z(H_1))=\E(Z(H_2))=\frac{1}{c^2}\left(1-\frac{1}{c}\right)^2, \quad \E(Z(H_4))=\frac{1}{c^3}\left(1-\frac{1}{c}\right).$$
Finally, $M(G_n, H_1)+M(G, H_2)= {|E(G_n)|\choose 2}\frac{4!}{2^2}$, which is obtained by choosing 2 out of the $|E(G_n)|$ edges and then permuting the $4$ edges (each chosen edge doubled) within themselves. By substituting in (\ref{lm:fourth_moment}),
\begin{eqnarray}
\E(Z_n^4)=\E\E(Z_n^4|G_n)=3\left(1-\frac{1}{c}\right)^2+\frac{1}{c}\left(1-\frac{1}{c}\right) \E\left(\frac{N(G_n, C_4)}{|E(G_n)|^2}\right).
\label{lm:fourth_moment_II}
\end{eqnarray}
By taking limits, $\limsup_{n\rightarrow\infty}\E(Z_n^4)>3\left(1-\frac{1}{c}\right)^2$.
%
Now, since $|\E Z(H)-\E W(H)|\le 1$ and $M(G_n,H)\lesssim_H |E(G_n)|^{\gamma(H)}$ it follows that
$$|\E(Z_n^8|G_n)-\E(W_n^8|G_n)|\lesssim_k |E(G_n)|^{\gamma(H)-|E(H)|/2}.$$
By Lemma \ref{lm:degree_one}, the last term is $O(1)$ as  $\gamma(H)=|E(H)|/2$, for all $H\in \mathcal{M}_k$, and so
\begin{equation}
\E(Z_n^8|G_n)\le \E(W_n^8|G_n)+C,
\label{eq:eighth_moment}
\end{equation} for a positive finite constant $C$. Finally, since
$\E e^{tW_n}<\infty$ for all $t\in \R$ (Lemma \ref{obvious2}), it follows that $\sup_{n\ge 1}\E W_n^8<\infty$, and so $\sup_{n\ge 1}\E Z_n^8<\infty$ from (\ref{eq:eighth_moment}). Now, suppose $Z_n\dto Z\sim N(0, 1-1/c)$. Since $\sup_{n\ge 1}\E Z_n^{8}<\infty$,  $Z_n^{4}$ is uniformly integrable. But then the convergence in law implies that $\E Z_n^{4}$  converges to $\E Z^4=3(1-1/c)^2$, which is a contradiction.

\subsection{Connections between the ACF4 and USN Conditions}
\label{sec:star}

Recall that for the case $c=2$, $N(G_n)$ can be rewritten as a quadratic form in terms of the adjacency matrix $A(G_n)$. In this case, the classical sufficient condition (\ref{eq:quadratic_form_condition}) for asymptotic normality of $V_n=\sum_{i\leq j}a_{ij}X_iX_j$, for $X_1, X_2, \ldots X_n$ i.i.d. with zero mean and finite fourth moment, can be re-written as:
\begin{equation}
\lim_{n\rightarrow \infty} \sigma_n^{-4}\E(V_n-\E(V_n))^4=3, \text{ and } \lim_{n\rightarrow \infty}\sigma_n^{-2} \max_{i \in [n]}\sum_{j=1}^n a_{ij}^2=\lim_{n\rightarrow \infty}\frac{\Delta(G_n)}{|E(K_{1, n})|}=0,
\label{eq:quadratic_form_graph}
\end{equation}
where $\Delta(G_n)$ is the maximum degree of a vertex in $G_n$. Moreover, from the proof of the necessity part of Theorem \ref{th:normal_fixed} in the previous section, it can be seen that the ACF4 condition is equivalent to the first condition in (\ref{eq:quadratic_form_condition}). 

For the case when  $X_1, X_2, \ldots X_n$ are i.i.d. Rademacher variables, Nourdin et al. \cite{rademacher}  showed that the second condition in (\ref{eq:quadratic_form_graph}) is not needed for asymptotic normality. Error bounds were also proved by Chatterjee \cite{chatterjee_normal_approximation} using Stein's method, under the conditions 
\begin{equation}
\lim_{n\rightarrow \infty} ||\pmb\lambda(A(G_n))||_4^4=0 \text{ and } \lim_{n\rightarrow \infty}\sigma_n^{-2} \max_{i \in [n]}\sum_{j=1}^n a_{ij}^2=0.
\label{eq:quadratic_form_condition_II}
\end{equation}
It is easy to see the first condition in (\ref{eq:quadratic_form_condition_II}) is equivalent to the USN condition $\lim_{n\rightarrow \infty} ||\pmb\lambda(A(G_n))||_\infty=0$.

Even though the number of cycles in a graph is closely related to the power sum of eigenvalues, there are subtle differences between ACF4 and USN conditions. To this end, consider the following example:

\begin{example} Consider the star graph $K_{1, n}$ with vertices indexed by $\{0, 1, \ldots, n\}$, with the central vertex labeled 0. It is easy to see that $K_{1, n}$ does not satisfy the USN condition: 
$$\vec \lambda(A(K_{1, n}))=(n^{\frac{1}{2}}, 0, \ldots, 0, -n^{\frac{1}{2}})', \text{ and } \lim_{n\rightarrow \infty} ||\pmb\lambda(A(K_{1, n}))||_\infty=\frac{1}{\sqrt{2}}\ne 0.$$ 
Moreover, $\Delta(K_{1, n})/|E(K_{1, n})|=1$, that is,  the second condition in (\ref{eq:quadratic_form_condition}) and (\ref{eq:quadratic_form_condition_II}) is also not satisfied by $K_{1, n}$. Therefore, asymptotic normality of $N(K_{1, n})$ does not follow from results of de Jong \cite{dejong} or Chatterjee \cite{chatterjee_normal_approximation}.

However, $K_{1, n}$ has no cycles and so the ACF4 condition is satisfied. Therefore, by Theorem  \ref{th:normal_fixed} (also by Nourdin et al. \cite{rademacher}) it follows that $N(K_{1, n})$ is asymptotically normal. In fact, for $\vec X=(X_0, X_1, X_2, \ldots X_n)'$ i.i.d. Rademacher 
$$n^{-\frac{1}{2}}\vec X' A(K_{1, n})\vec X=X_0\left(\frac{\sum_{i=1}^n X_i}{\sqrt n}\right)\dto X_0\cdot Z\sim N(0, 1),$$
where $Z\sim N(0, 1)$.

Consider on the other hand, $\vec Z=(Z_0, Z_1, Z_2, \ldots Z_n)'$ i.i.d. $N(0, 1)$. Then 
$$n^{-\frac{1}{2}}\vec Z A(K_{1, n})\vec Z=Z_0\left(\frac{\sum_{i=1}^n Z_i}{\sqrt n}\right)\stackrel{\sD}= Z_0 S_0$$
where $Z_0, S_0\sim N(0, 1)$, which is not normally distributed. This is expected because (\ref{eq:quadratic_form_condition}) is necessary and sufficient when each vertex in the graph is assigned an independent normal random variable.
\end{example}

\subsection{Examples}

Some examples where the ACF4 condition is satisfied and asymptotic normality holds are illustrated below.

\begin{example}(Random Regular Graphs) $\sG_n$ consists of the set all $d$-regular graphs on $n$ vertices and sampling is done uniformly on this space. In this case, $|E(G_n)|=nd/2$, for all $G_n\in \sG_n$. Theorem \ref{th:normal} gives
\begin{equation}
\frac{N(G_n)-\frac{1}{c}\left(\frac{nd}{2}\right)}{\sqrt{\frac{nd}{2c}\left(1-\frac{1}{c}\right)}} \stackrel{\sD}\rightarrow N(0, 1) \text{ when } \frac{nd}{c}\rightarrow \infty.
\label{eq:dregular_normal}
\end{equation}

Moreover, for all $G_n\in \sG_n$, $\lambda_{\max}(G_n)=d$ and $\sum_{i=1}^n\lambda_i(G_n)^2=2|E(G_n)|=nd$. This implies that $||\pmb \lambda_n||_\infty=O(\sqrt{d/n})=o(1)$, whenever $d=o(n)$.
Therefore, by Observation \ref{obs:usn} and Theorem \ref{th:normal_fixed}, (\ref{eq:dregular_normal}) holds even when $c$ is fixed and $d=o(n)$.
\end{example}

\begin{example}(Hypercube) An important $d$-regular graph which is neither sparse nor dense is the hypercube $Q_n=(V(Q_n), E(Q_n))$, where $V(Q_n)=\{0, 1\}^{\log_2 n}$, where $n=2^s$ for some $s\geq 2$, and there exists an edge between two vertices if the corresponding binary vectors have Hammimg distance 1. This is a $d$-regular graph with $d=\log_2 n$ and $|E(Q_n)|=\frac{1}{2}n \log_2 n$. Therefore,
$||\pmb \lambda_n||_\infty=O(\sqrt{d/n})=o(1)$ and by the previous example,
 \begin{equation}
\frac{N(G_n)-\frac{1}{c}\left(\frac{n\log_2n}{2}\right)}{\sqrt{\frac{n\log_2 n}{2c}\left(1-\frac{1}{c}\right)}} \stackrel{\sD}\rightarrow N(0, 1) \text{ when } \frac{n \log_2 n}{c}\rightarrow \infty.
\label{eq:hypercube_normal}
\end{equation}
irrespective of whether $c\rightarrow \infty$ or $c$ is fixed.
\end{example}

\begin{example}(Sparse Inhomogeneous Random Graphs) A general model for sparse random graphs is the following: every edge $(i, j)$ is present independently with probability $\frac{1}{n}\cdot f(\frac{i}{n}, \frac{j}{n})$, for some symmetric continuous function $f:[0,1]^2\rightarrow [0,1]$ (see Bollobas et al. \cite{bollobasjansonriordan}). By the law of large numbers,
$$\frac{1}{n}|E(G_n)|\stackrel{\sP}{\rightarrow} \frac{1}{2}\int_0^1\int_0^1f(x, y)dxdy,$$
and it is easy to see that 
$$\lim_{n\rightarrow \infty}\E N(G_n, C_4)=\frac{1}{2} \int_{[0, 1]^4}f(w, x)f(x, y)f(y, z)f(z, x)dwdxdydz.$$
Therefore by Markov's inequality   $$\frac{N(G_n, C_4)}{|E(G_n)|^2}\stackrel{\sP}{\rightarrow} 0$$ and so 
by Theorem \ref{th:normal} and Theorem \ref{th:normal_fixed} we have 
\begin{equation}
\frac{N(G_n)-\frac{|E(G_n)|}{c}}{\sqrt{\frac{n}{c}\left(1-\frac{1}{c}\right)}} \stackrel{\sD}\rightarrow N\left(0, \frac{1}{2}\int_0^1\int_0^1f(x, y)dxdy\right) \text{ when } \frac{n}{c}\rightarrow \infty,
\label{eq:dregular_normal}
\end{equation}
irrespective of whether $c\rightarrow \infty$ or $c$ is fixed. Note that this model includes as a special case the Erd\H os-Renyi random graphs $G(n, \lambda/n)$, by taking the function $f(x, y)=\lambda$.
\end{example}

\section{Universal Non-Normal Limit for Dense Graphs}
\label{sec:chisquare}

The precise conditions required for the normality of the number of monochromatic edges $N(G_n)$ have been determined in the previous sections. It is also in Example \ref{ex:K2n} that when these conditions are not met, $N(G_n)$ might have non-standard limiting distributions. However, in this section it will be shown that there is a universal characterization of the limiting distribution of $N(G_n)$ for a converging sequence of dense graphs $G_n$.

To this end, consider the following example where the limiting distribution of the number of monochromatic edges is determined for a complete graph.

%
%

\begin{example}(Complete Graph) Consider coloring the complete graph $K_n$ with $c=2$ colors under the uniform distribution, where $c$ is fixed.  Let $N(K_n)$ be the number of monochromatic edges of the complete graph $K_n$. Let $U_n$ be the proportion of vertices of $K_n$ with color $1$. Then 
$$nU_n\sim \mathrm{Binomial}(n,1/2), \text{ and } n^{1/2}(U_n- 1/2)\stackrel{\sD}\rightarrow N(0, 1/4).$$
In this case, we have $N(K_n)={nU_n\choose 2}+{n-nU_n\choose 2}=\frac{n^2}{2}\left(U_n^2+(1-U_n)^2\right)-\frac{n}{2}$,
and so
$$\frac{1}{n}\left(N(K_n)-\frac{1}{2}{n \choose 2}\right)=\frac{n}{2}\left(U_n^2+(1-U_n)^2-\frac{1}{2}\right)-\frac{1}{4}\stackrel{\sD}\rightarrow \frac{1}{4}(\chi^2_{(1)}-1),$$
where the last convergence follows by an application of Delta method to the function $f(x)=x^2+(1-x)^2$, and noting that $f'(1/2)=0, f''(1/2)=4.$
\label{ex:completegraph}
\end{example}

This example motivates the characterization of the limiting distribution for any converging sequence of dense graphs. The limit theory of dense graphs was developed recently by Lov\'asz and coauthors \cite{graph_limits_I,graph_limits_II,lovasz_book}. Using results from this limit theory, the limiting distribution of $N(G_n)$ can be characterized for any dense graph sequence $G_n$ converging to a limit $W\in\sW$.

\subsection{Proof of Theorem \ref{th:chisquare}}

Write $$N(G_n)-\frac{|E(G_n)|}{c}=\sum_{(i, j)\in E(G_n)}\sum_{a=1}^c\left(\pmb1\{Y_i=a\}-\frac{1}{c}\right)\left(\pmb1\{Y_j=a\}-\frac{1}{c}\right).$$
As before, Theorem \ref{th:chisquare} will be proved by comparing the conditional moments of 
\begin{equation}\label{eq:gamma_n}
\Gamma_n=\frac{N(G_n)-\frac{|E(G_n)|}{c}}{\sqrt{2|E(G_n)|}}
\end{equation}
with another random variable for which the asymptotic distribution can be obtained much easily.
To this end, define the random variable $$Q(G_n):=\sum_{(i, j)\in E(G_n)}\sum_{a=1}^cS_{ia}S_{ja},$$ with $S_{va}=X_{va}-\overline{X}_{v.}$, where $\{X_{va}:v\in V(G_n), a\in [c]\}$ is a collection of independent $N(0,1/c)$ random variables and $\overline{X}_{v.}=\frac{1}{c}\sum_{a=1}^cX_{va}$.
Note that for each $v\in V(G_n)$ the vector $\vec S_v:=(S_{v1}, S_{v2}, \ldots, S_{vc})'$ is a multivariate normal conditioned on $\sum_{a=1}^cS_{va}=0$. Also, $\{\vec S_v, v\in V(G_n)\}$ are independent and identically distributed random variables.

Finally, define \begin{equation}\label{eq:tn}\Delta_n:=\frac{Q(G_n)}{\sqrt{2|E(G_n)|}}.
\end{equation}

\subsection{Comparing Conditional Moments}

The moments of $\Gamma_n$ and $\Delta_n$ involve sum of multi-subgraphs of $G_n$. We begin with a simple observation about general multi-graphs.

\begin{obs}
Let $H=(V(H), E(H))$ be any finite multigraph with $d_{\min}(H)\geq 2$ and $|V(H)|=|E(H)|$. Then $H$ is a disjoint union of cycles and isolated doubled edges.
\label{obs:cycle_edge_vertex}
\end{obs}

\begin{proof}
Let $H_1, H_2\ldots, H_\nu$ be the connected components of $H$. Note that if there exists $i \in [\nu]$ such that $|E(H_i)|<|V(H_i)|$, then $H_i$ must be a tree, which has a vertex of degree 1. Therefore, $|E(H_i)|=|V(H_i)|$ for all $i \in [\nu]$.

Now, let $F$ be any connected component of $H$, and $F_S$ be the underlying simple graph. Since $F$ is connected either $|E(F_S)|=|V(F)|=|E(F)|$ or $|E(F_S)|=|V(F)|-1=|E(F)|-1$.

If $|E(F_S)|=|V(F)=|E(F)|$, then $F$ itself is a simple graph with  $d_{\min}(F)\geq 2$, which implies that $F$ is a cycle of length $|V(F)|$.

On the other hand, if $|E(F_S)|=|V(F)|-1=|E(F)|-1$, then $F_S$ is a tree. But any tree has at least two degree one vertices, and one extra multi-edge cannot add to both their degrees unless the tree is just an isolated edge. This implies that $F_S$ is an isolated edge, and $F$ is an isolated doubled edge.
\end{proof}

As in Lemma \ref{lm:moment_comparison}, the following lemma identifies the set of multi-graphs for which the moments are equal.

\begin{lem}For any multi-subgraph $H=(V(H), E(H))$ of $G_n$ define
$$Z(H)=\prod_{(i,j)\in E(H)}\sum_{a=1}^c\left(\pmb 1\{Y_i=a\}-\frac{1}{c}\right)\left(1\{Y_j=a\}-\frac{1}{c}\right), \text{ and } T(H)=\prod_{(i,j)\in E(H)}\sum_{a=1}^cS_{ia}S_{ja}.$$
If $|V(H)|=|E(H)|$, then $\E(Z(H)) =\E(T(H))$.
\label{lm:moment_equal}
\end{lem}

\begin{proof}
If $d_{\min}(H)=1$, by arguments similar to Lemma \ref{lm:degree_one}, $\E(Z(H)) =\E(T(H))=0$.
Therefore, it suffices to assume that $d_{\min}(H)\geq 2$ and $|V(H)|=|E(H)|$. By Observation \ref{obs:cycle_edge_vertex}, $H$ is a disjoint union of cycles and isolated doubled edges.
Since both $Z(H)$ and $T(H)$ factorize over connected components, w.l.o.g. $H$ can be assumed to be either an isolated doubled edge or a cycle. More generally, it suffices to show that $\E(T(H))=\E(Z(H))$ for any multigraph $H$ with each vertex having degree $2$. Now, it is easy to see that since the random variables corresponding to each vertex are independent in both the cases, it suffices to prove
$$\E \left(\pmb1\{Y_i=a\}-\frac{1}{c}\right)\left(\pmb1\{Y_i=b\}-\frac{1}{c}\right)=\E S_{ia}S_{ib},$$
for any $a, b\in [c]$. This follows on noting that if $a= b$ both sides above equal $\frac{1}{c}(1-\frac{1}{c})$, whereas for $a\neq b$ both sides above equal $-\frac{1}{c^2}$.
\end{proof}

Using this lemma it can now be shown that the conditional moments of $\Gamma_n$ and $\Delta_n$ are asymptotically close, whenever the graph sequence $G_n$ converges to $W\in \sW$ such that $\int_{[0, 1]^2}W(x, y)dxdy>0$, that is, $G_n$ is dense.

\begin{lem}
Suppose the sequence of random graphs $\{G_n\}_{n=1}^\infty$ converges in  distribution to a limit $W\in \sW$ with $\int_{[0, 1]^2}W(x, y)dxdy>0$ almost surely. Then for $\Gamma_n$ and $\Delta_n$ as defined in (\ref{eq:gamma_n}) and (\ref{eq:tn}), and for every fixed $k\geq 1$,
$$|\E (\Delta_n^k|G_n)-\E (\Gamma_n^k|G_n)|\stackrel{\sP}{\rightarrow}0.$$
\label{lm:chisquare_difference_zero}
\end{lem}

\begin{proof}By Equation (\ref{eq:zn}),
\begin{eqnarray}
|\E (\Delta_n^k|G_n)-\E(\Gamma_n^k|G_n)|&\leq&
\sum_{\substack{H\in \cH_k\\ }}\frac{M(G_n, H)}{(2|E(G_n)|)^{\frac{k}{2}}}\left|\E(T(H))-\E(Z(H))\right|,
\label{eq:zn_chisquare}
\end{eqnarray}
where $\cH_k$ is the collection of all multi-subgraphs of $G_n$ with exactly $k$ edges and no isolated vertex.
If $H\in \cH_k$ is such that $|V(H)|>|E(H)|$, then $H$ must have an isolated vertex and $\E(T(H))=\E(Z(H))=0$. Moreover, if  $H\in \cH_k$ is such that $|V(H)|=|E(H)|$, then by Lemma \ref{lm:moment_equal} $\E(T(H))=\E(Z(H))$. Therefore, (\ref{eq:zn_chisquare}) simplifies to
\begin{eqnarray*}
|\E (\Delta_n^k|G_n)-\E(\Gamma_n^k|G_n)|&\lesssim_k&\sum_{\substack{H\in \mathcal{H}_k\\|V(H)|<|E(H)|}}\frac{N(G_n, H)}{(2|E(G_n)|)^{\frac{k}{2}}}\lesssim_k \sum_{\substack{H\in \mathcal{H}_k\\|V(H)|<|E(H)|}}\frac{|V(G_n)|^{|V(H)|}}{(2|E(G_n)|)^{\frac{|E(H)|}{2}}},
\end{eqnarray*}
where the last term follows from noting that $N(G_n, H)\lesssim_H |V(G_n)|^{|V(H)|}$ for any $H$ and $G_n$. Now, since $\int_{[0,1]^2}W(x,y)dxdy>0$ almost surely and $$\frac{2|E(G_n)|}{|V(G_n)|^2\int_{[0,1]^2}W(x,y)dxdy}\stackrel{\sP}{\rightarrow 1},$$ it follows that $\frac{|V(G_n)|^{|V(H)|}}{(2|E(G_n)|)^{\frac{|E(H)|}{2}}}=O_P(|V(G_n)|^{|V(H)|-|E(H)|})$, which  goes to zero in probability for all
$H\in \mathcal{H}_k$ such that $|V(H)|<|E(H)|$. This completes the proof of the lemma.
\end{proof}

\subsubsection{Completing the Proof of Theorem \ref{th:chisquare}}

As the conditional moments of $\Gamma_n$ and $\Delta_n$ are asymptotically close, it remains to analyze the limiting distribution of $\Delta_n$. In this section it will be shown that $\Delta_n$ converges to $\frac{1}{2c}\chi^2_{c-1}(W)$, where $$\chi_{c-1}(W):=\sum_{i=1}^\infty\left(\frac{\lambda_i(W)}{(\sum_{j=1}^\infty\lambda^2_j(W))^{\frac{1}{2}}}\right)\xi_i,$$
where $\{\xi_i\}_{i\in \N}$ are independent $\chi_{(c-1)}-(c-1)$ random variables.

The first step is to show that the random variable $\chi^2_{c-1}(W)$, which is a infinite sum of centered chi-square random variables, is well defined.

\begin{ppn}\label{ppn:chisquare_define}
Let $(a_1,a_2,\cdots)$ be an infinite sequence of random variables such that $\sum_{j=1}^\infty a_j^2\stackrel{a.s.}{=}1$. Given $(\xi_1,\xi_2,\cdots)$ independent $\chi^2_c-c$ random variables independent of the sequence $(a_1,a_2,\cdots)$, the sum
$S:=\sum_{j=1}^\infty a_j\xi_j$ converges almost surely and in $L^1$. Further, for $|t|<1/8$ the moment generating function of $S$ is finite, and is given by
$$\E e^{t S}=\E \left(\prod_{j=1}^\infty\frac{1}{(1-2t a_j)^{\frac{c}{2}}}\right).$$
\end{ppn}

\begin{proof}
By defining $\sG:=\sigma(\{a_j\}_{j\in \N})$ and $S_n:=\sum_{j=1}^na_j\xi_j$ and $\sF_n:=\sigma(\sigma(\{\xi_j\}_{j=1}^n),\sG)$, it follows that $(S_n,\sF_n)$ is a martingale, with
$$\limsup_{n}\E S_n^2=2c\left(\E \sum_{j=1}^\infty a_j^2\right)=2c<\infty,$$
and $S_n$ converges almost surely and in $L^1$ \cite{durrett}.

To compute the moment generating function, first note that $h: [-1/2,1/2]\mapsto \R$ given by $h(z)=-\log (1-z)-z- z^2$ has a unique global maxima at $z=0$, and so $-\log(1-z)-z\le z^2$ for $|z|\le 1/2$. Therefore, for any $|t|< 1/8$,
\begin{align*}
\log\E (e^{2t S_n}|\sG)=c\sum_{j=1}^n\left(\frac{-\log(1-4t a_j)-4t a_j
}{2}\right)\le 8ct^2,
\end{align*}
and $\E (e^{2t S_n}|\sG)\le e^{8ct ^2}<\infty.$
It follows that $e^{t S_n}$ and $\E(e^{t S_n}|\sG)$ are both uniformly integrable, and
\begin{align*}
\E e^{t S}=  \lim_{n\rightarrow\infty}\E\left(\E( e^{t S_n}|\sG)\right)
=\lim_{n\rightarrow\infty}\E \left(\prod_{j=1}^n\frac{1}{(1-2t a_j)^{\frac{c}{2}}}\right)=\E \left(\prod_{j=1}^\infty\frac{1}{(1-2t a_j)^{\frac{c}{2}}}\right)<\infty.
\end{align*}
\end{proof}

To prove Theorem \ref{th:chisquare} we now invoke Lemma \ref{abc} with $\Gamma_n$ and $\Delta_n$. Lemma \ref{lm:chisquare_difference_zero} shows that (\ref{eq:abc1}) holds. The following lemma takes the first step towards (\ref{eq:abc2}) by showing that the limiting distribution of $\Delta_n$ is a weighted sum of chi-square random variables.

\begin{lem}If a sequence of graphs $\{G_n\}_{n=1}^\infty$ converges in the cut-metric to a limit $W\in \sW$, then  for $|t|<c/4$
$$\lim_{n\rightarrow\infty}\E e^{t \Delta_n}=\E e^{t \frac{\chi^2_{c-1}(W)}{2c}}.$$
\end{lem}

\begin{proof}
Using the spectral decomposition, write the adjacency matrix $A(G_n)$ as $\sum_{j=1}^n\lambda_j(G_n)\vec p_j \vec p_j'$. For $a\in [c]$ and $j\in [|V(G_n)|]$, set $\pmb S_a:=(S_{1a}, S_{2a}, \cdots, S_{|V(G_n)|a})'$ and $y_{aj}=\vec p_j'\pmb S_a$. Then
\begin{align*}
\sum_{a=1}^c\pmb S_{a}'A(G_n)\pmb S_a=\sum_{j=1}^{|V(G_n)|}\lambda_j(G_n)\sum_{a=1}^cy_{aj}^2.
\end{align*}
Since $\Cov(y_{aj},y_{bj})=-1/c^2$ for $1\le a< b\le c$, it follows that $A_j:=\sum_{a=1}^cy_{aj}^2\sim \frac{1}{c}\chi^2_{c-1}$. Also, since
$\Cov(y_{ai},y_{aj})=0$ for $i\neq j$, it follows that $\{A_j\}_{j=1}^n$ are i.i.d., and
\begin{align*}
\E\left(e^{t \Delta_n}\Big|G_n\right)=\E \left(\exp\left\{\frac{t\sum_{j=1}^n \lambda_j(G_n)A_j}{2\sqrt{\sum_{j=1}^n{\lambda^2_j(G_n)}}}\right\}\Big|G_n\right)
=\prod_{j=1}^n\left(1-\tilde\lambda_j(G_n)\frac{t}{c}\right)^{\frac{1-c}{2}}
\end{align*}
with $\tilde\lambda_j(G_n):=\frac{\lambda_j(G_n)}{||\vec \lambda(G_n)||_2}.$

Now, since $G_n\stackrel{\sD}{\Rightarrow} W$, $$\{t(C_g, G_n),g\ge 3\}\stackrel{\sD}{\rightarrow} \{t(C_g, W),g\ge 3\} $$
on $[0,1]^{\N}$
with $t(C_g,W)=\sum_{i=1}^\infty \lambda^g_i(W)$, where $\{\lambda_j(W)\}_{j\in\mathbb{N}}$ are the eigenvalues of $W$ (see \cite[Section 7.5]{lovasz_book}). Therefore,
\begin{align*}
\log\prod_{j=1}^n\left(1-\tilde\lambda_j(G_n)\frac{t}{c}\right)^{\frac{1-c}{2}}=&\left(\frac{c-1}{2}\right)\sum_{x=2}^\infty\frac{t^x}{xc^x}\cdot\frac{t(C_g,G_n)}{t(C_2,G_n)^{\frac{x}{2}}}\\
\stackrel{\sD}{\rightarrow}&\left(\frac{c-1}{2}\right)\sum_{x=2}^\infty\frac{t^x}{xc^x}\cdot\frac{t(C_g,W)}{t(C_2,W)^{\frac{x}{2}}}\\
=&\sum_{j=1}^\infty \log\left(1-\frac{\lambda_j(W)}{\sqrt{\sum_{j=1}^\infty \lambda_j^2(W)}}\cdot\frac{t}{c}\right)^{\frac{1-c}{2}},
\end{align*}
where the first equality uses $\sum_{j=1}^n\tilde\lambda_j(G_n)=0$.
Finally, as $$\prod_{j=1}^n\left(1-\tilde\lambda_j(G_n)\frac{t}{c}\right)^{\frac{1-c}{2}}\le \left(1-\frac{t}{c}\right)^{\frac{1-c}{2}}<\infty,$$ it
follows by uniform integrability that
$$\E e^{t \Delta_n}=\E\E \left(e^{t \Delta_n}\Big|G_n\right)\rightarrow \E \prod_{j=1}^\infty \left(1-\frac{\lambda_j(W)}{\sqrt{\sum_{j=1}^\infty \lambda_j^2(W)}}\cdot\frac{t}{c}\right)^{\frac{1-c}{2}}. $$
By Proposition \ref{ppn:chisquare_define}, the RHS is the moment generating function of $\frac{\chi^2_{c-1}(W)}{2c}$, and the proof is complete.
\end{proof}


Finally, to complete the proof of Theorem \ref{th:chisquare} using Lemma \ref{abc}, observe: since $\E e^{t \Delta_n}$ converges to $\E e^{t\frac{\chi^2_{c-1}(W)}{2c}}$, for some $t >0$, the arguments identical to those used in (\ref{eq:exponential_moment}) imply condition (\ref{eq:abc2}) for $\Delta_n$.


\subsection{More Examples}

The limiting chi-square distribution of the complete graph was illustrated before in Example \ref{ex:completegraph}. A few other simple examples where Theorem \ref{th:chisquare} can be used to determine the limiting distributions are given below.

\begin{example}(Complete Bipartite Graph) Consider the complete bipartite graph $K_{\frac{n}{2}, \frac{n}{2}}$, which converges to the limit $W(x, y)=\pmb 1\{(x-1/2)(y-1/2)<0\}$. It is easy to see that the only non-zero eigenvalues of $W$ are $\pm \frac{1}{2}$. Therefore, by Theorem \ref{th:chisquare} it follows that
$$ \frac{1}{n}\left(N(K_{\frac{n}{2}, \frac{n}{2}})-\frac{n^2}{4c}\right)\stackrel{\sD}\rightarrow \frac{1}{4c}\left(\xi_1-\xi_2\right),$$
where $\xi_1$ and $\xi_2$ are independent $\chi^2_{(c-1)}$ random variables.
\end{example}

\begin{example}(Inhomogeneous Random Graphs)
Let $f:[0,1]^2\rightarrow [0,1]$ be a symmetric continuous function. Consider the random graph model where an edge $(i, j)$ is present with probability $f(\frac{i}{n}, \frac{j}{n})$.
Therefore, whenever $\int_{[0, 1]^2}f(x, y)dxdy>0$ the limit theorem in (\ref{eq:chisquare}) holds.
In particular, the Erd\H os-Renyi random graph $G(n, p)$ can be obtained by taking the function $f(x, y)=p$. In this case, $p$ is the only non-zero eigenvalue of the operator $f$ and
$$ \frac{1}{n}\left(N(G(n, p))-\frac{|E(G_n)|}{c}\right)\stackrel{\sD}\rightarrow \frac{p^{\frac{1}{2}}}{2c}\cdot \left(\chi^2_{(c-1)}-(c-1)\right).$$
Note that this reduces to Example \ref{ex:completegraph}, for $c=2$ and $p=1$.
\end{example}


\section{Extremal Examples: Stars and Cycles}
\label{sec:starscycles}


Another relevant question is whether it is possible to expect a similar Poisson universality result for other subgraphs, under the uniform coloring scheme? This section begins by proving Proposition \ref{propn:easy} which shows that we may not get Poisson mixtures in the limit while counting monochromatic $r$-stars, in a uniform $c$-coloring of an $n$-star.

\subsection{Monochromatic Stars}

Consider the $(n-1)$-star, $K_{1, n-1}$ with vertices labelled by $[n]$, with the central vertex labeled 1. Color the vertices of $K_{1, n-1}$, uniformly from $[c]$, independently. Consider the limiting distribution of the number of monochromatic $r$-stars $K_{1, r-1}$ generated by this random coloring, where $r$ is a fixed constant. If $Y_i$ denotes the color of vertex $i$, the random variable is
$$T_{r, n}=\sum_{\substack{S\subseteq [n]\backslash \{1\}\\|S|=r-1}} \prod_{j \in S}\pmb \{Y_1=Y_j\}.$$

Proposition \ref{propn:easy} shows that the limiting behavior of $T_{r, n}$ cannot converge to a mixture of Poissons. This illustrates that the phenomenon of universality of the Poisson approximation that holds for the number of monochromatic edges, does not extend to arbitrary subgraphs. In particular, it is not even true for the 2-star, which is the simplest extension of an edge.

\subsubsection{Proof of Proposition \ref{propn:easy}}Note that if the number of monochromatic edges in $G_n=K_{1, n-1}$ is $N(G_n)$, then
$$T_{r, n}\stackrel{\sD}{=}{N(G_n)\choose r}.$$
If $n/c\rightarrow 0$, then from Theorem \ref{th:poissonuniversal} $N(G_n)\stackrel{\sP}\rightarrow 0$ and so $T_{r, n}\stackrel{\sP}\rightarrow 0$. Similarly, if $n/c\rightarrow \infty$, $T_{r, n}\stackrel{\sP}\rightarrow \infty$.

Finally, if $\frac{n}{c}\rightarrow \lambda$, the number of monochromatic edges $N(G_n)$ in $K_{1, n-1}$ converges to $X\sim Poisson(\lambda)$, by Theorem \ref{th:poissonuniversal}. This implies that $$T_{r, n}\stackrel{\sD}{=}{N(G_n)\choose r}\stackrel{\sD}{\rightarrow} {X\choose r}=\frac{X(X-1)\cdots(X-r+1)}{r!}.$$
This random variable does not assign positive masses at all non-negative integers, and so it cannot be a mixture of Poisson variates.

\subsection{Monochromatic Cycles}

\begin{figure*}[h]
\centering
\begin{minipage}[c]{0.4\textwidth}
\centering
\includegraphics[width=2.8in]
    {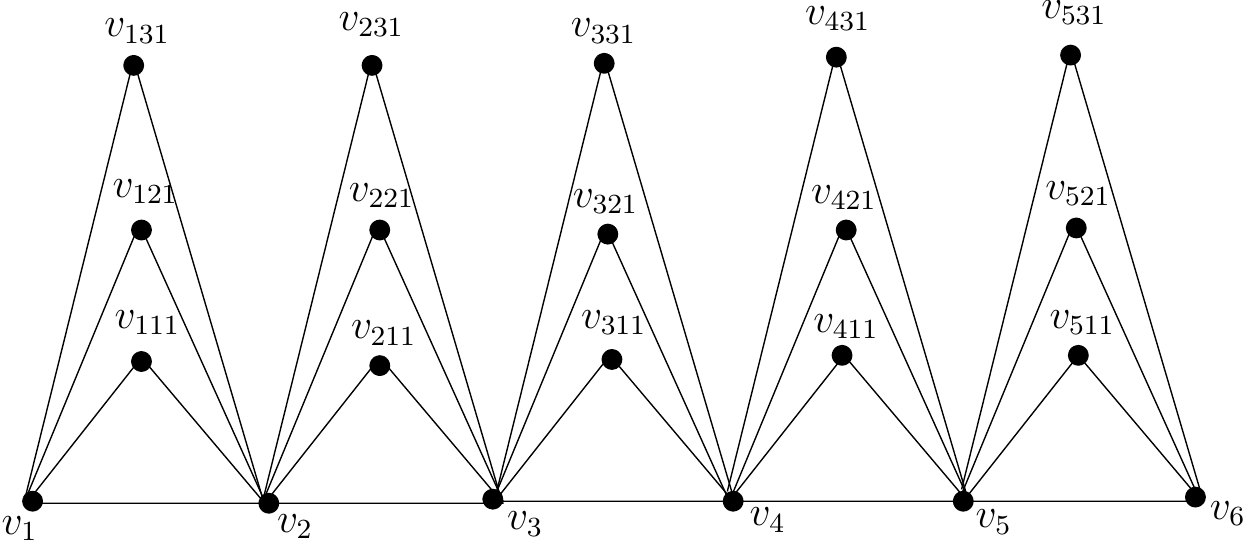}\\
\small{(a)}
\end{minipage}
\begin{minipage}[c]{0.5\textwidth}
\centering
\includegraphics[width=3.2in]
    {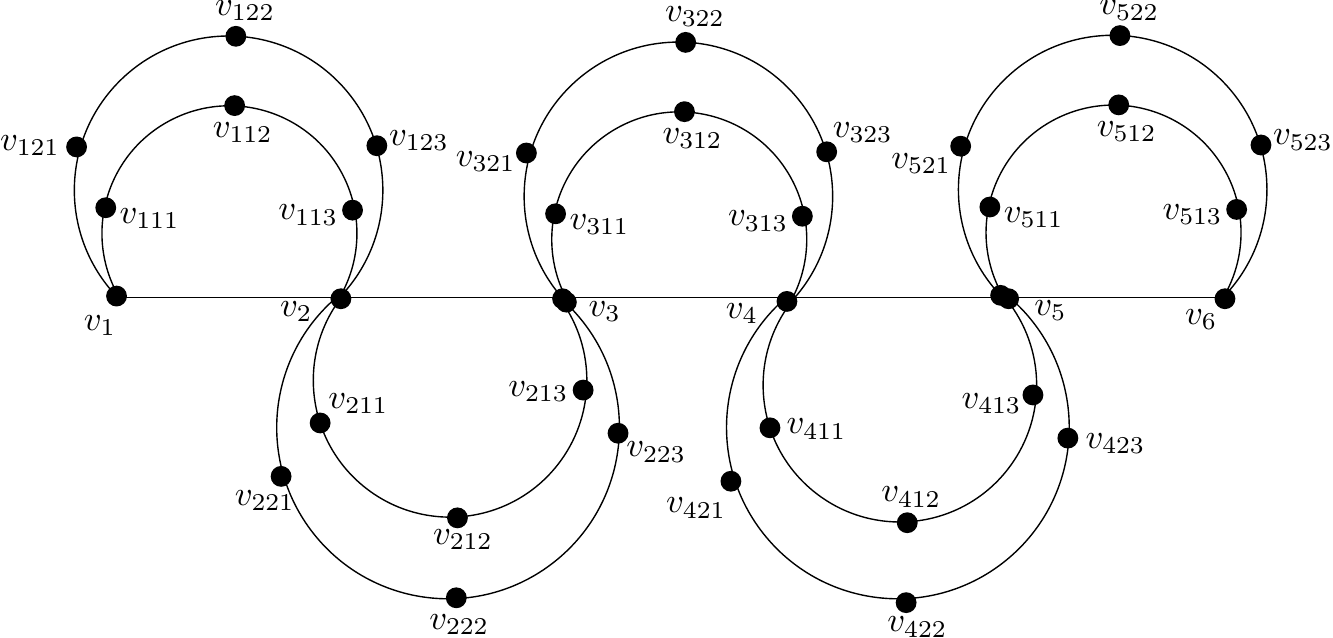}\\
\small{(b)}
\end{minipage}
\caption{(a) The graph $G_{5,3,3}$, and (b) the graph $G_{5,2,5}$}
\label{fig:pathcyclegraph}
\end{figure*}

Recall that the number of monochromatic edges $N(G_n)$ converges to $Poisson(\lambda)$ whenever $|E(G_n)|/c\rightarrow \lambda$. The limiting distribution of the number of edges can only be a non-trivial mixture of Poissons when $|E(G_n)|/c\rightarrow Z$, and $Z$ has a non-degenerate distribution. We now construct a graph $G_n$ where the number of monochromatic $g$-cycles in a uniform $c$-coloring of $G_n$ converges in distribution to a non-trivial mixture of Poissons even when $|N(G_n, C_g)|/c^{g-1}$ converges to a fixed number $\lambda$. This phenomenon, which cannot happen in the case of edges, makes the problem of finding the limiting distribution of the number of monochromatic cycles,  much more challenging.

For $a, b$ positive integers and $g\ge 3$, define a graph $G_{a, b, g}$ as follows: Let $$V(G_{a, b, g})=\{v_1, v_2, \ldots v_{a+1}\}\bigcup_{i=1}^a\bigcup_{j=1}^b \{v_{ijk}: k \in \{1, 2, \ldots g-2\}\}.$$ The edges are such that vertices $v_1, v_2, \ldots v_{a+1}$ form a path of length $a$, and for every $i \in [a]$ and $j \in[b]$, $v_i, v_{ij1}, v_{ij2}, \ldots v_{ijg-2}, v_{i+1}$ form a cycle of length $g$ (Figure \ref{fig:pathcyclegraph} shows the structure of graphs $G_{5,3,3}$ and $G_{5, 2, 5}$, and the corresponding vertex labelings.). Note that graph $G_{a, b, g}$ has $b(g-2)+a+1$  vertices, $b(g-1)+a$ edges, and $ab$ cycles of length $g$.

We consider a uniform $c$-coloring of the vertices of $G_{a, b, g}$ and count the number of monochromatic $g$-cycles. Let $Y_i$ be the color of vertex $v_i$ and $Y_{ijk}$ the color of vertex $v_{ijk}$, for $i \in [a]$ and $j\in [b]$. The random variable $$Z_{a,b, g}:=Z(G_{a,b, g}):=\sum_{i=1}^a \sum_{j=1}^b \prod_{k=1}^{g-2}\pmb 1\{Y_{i}=Y_{i+1}=Y_{ijk}\},$$
which counts the number of monochromatic $g$-cycles in the graph $G_{a,b, g}$. The following proposition shows that there exists a choice of parameters $a, b, c$ such that $|N(G_n, C_g)|/c^{g-1}\rightarrow \lambda$ and $Z_{a,b, g}$ converges in distribution to a non-trivial mixture of Poissons.

\begin{ppn}For $a=\lambda n$ and $b=n^{g-2}$ and $c=n$, $Z_{a,b, g}\stackrel{\sD}\rightarrow Poisson(W)$, where $W\sim Poisson(\lambda)$.
\end{ppn}

\begin{proof}Let $\vec{Y}=(Y_1, Y_2, \cdots Y_{a+1})$ and note that
$$\prod_{k=1}^{g-2}\pmb 1\{Y_{i}=Y_{i+1}=Y_{ijk}\}\Big| \vec Y \sim \dBer(1/c^{g-2}) \text{ and }\sum_{j=1}^b \prod_{k=1}^{g-2}\pmb 1\{Y_{i}=Y_{i+1}=Y_{ijk}\} \Big| \vec Y\sim \dBin(b, 1/c^{g-2}).$$
This implies that
\begin{equation}
\E\left(e^{itZ_{a, b, g}}\right)=\E\left(\prod_{i=1}^a\E\left(e^{it \sum_{j=1}^b \prod_{k=1}^{g-2}\pmb 1\{Y_{i}=Y_{i+1}=Y_{ijk}\}}\Big|\vec Y\right)\right)=\E\left(1-\frac{1}{c^{g-2}}+\frac{e^{it}}{c^{g-2}}\right)^{bN_a},
\label{eq:chfcycle}
\end{equation}
where $N_a=\sum_{i=1}^a\pmb 1\{Y_i=Y_{i+1}\}$, is the number of monochromatic edges in the  path $v_1, v_2, \ldots, v_{a+1}$.

Substituting $a=\lambda n:=a_n$ and $b=n^{g-2}:=b_n$ and $c=n:=c_n$, we have $N(G, C_g)/c_n^{g-1}=a_nb_n/c_n=\lambda$. With this choice $a_n,b_n, c_n$, we have by Theorem \ref{th:birthdayuniformuniversal}, $N_{a_n}$ converges in distribution to $W:=Poisson(\lambda)$, as $a_n/c_n=\lambda$. Therefore,
$$\left(1-\frac{1}{c_n^{g-2}}+\frac{e^{it}}{c_n^{g-2}}\right)^{b_n N_{a_n}}=e^{b_n N_{a_n}\log\left(1-\frac{1}{c_n^{g-2}}+\frac{e^{it}}{c_n^{g-2}}\right)}\stackrel{\sD}\rightarrow e^{(e^{it}-1)W}.$$
By the dominated convergence theorem we have
$$\E\left(e^{itZ_{a_n, b_n, g}}\right)\stackrel{\sD}\rightarrow \E\left(e^{(e^{it}-1)W}\right),$$
which the characteristic function of $Poisson(W)$, where $W\sim Poisson(\lambda)$.
\end{proof}

\begin{remark}
We were unable to construct an example of a graph for which the number of monochromatic triangles converges to some distribution that is not a mixture of Poissons, when $N(G_n, C_3)/c^2\rightarrow \lambda$.  The above construction and the inability to construct any counterexamples, even for triangles, lead us to believe that some kind of Poisson universality holds for cycles as well. More formally, we conjecture that the number of monochromatic triangles in a uniform random coloring of any graph sequence $G_n$ converges in distribution to a random variable that is a mixture of Poissons, whenever $|N(G_n, C_3)|/c^{2}\rightarrow \lambda$, for some fixed $\lambda>0$.
\end{remark}

\noindent{\bf Acknowledgement:} The authors are indebted to Noga Alon for pointing out an error in the proof of Theorem \ref{th:alon_cycle} in an earlier draft. The authors also thank Sourav Chatterjee, Xiao Fang, and Bjarne Toft for helpful comments and suggestions. The authors further thank the anonymous referees for valuable comments.

\section{Appendix: Conditional Convergence to Unconditional Convergence}


There are many conditions on modes of convergence which ensure the convergence of a sequence of joint distributions when it is known that the associated sequence of marginal and conditional distributions converge \cite{sethuraman,sweeting}. This section gives a proof of a technical lemma which allows conclusions about the limiting distribution of a random variable from its conditional moments. The lemma is used in this paper in the final steps of our proofs of all the main theorems. 

\begin{lem}\label{abc}
Let $(\Omega_n,\sF_n,\mathbb{P}_n)$ be a sequence of probability spaces, and $\sG_n\subset \sF_n$ be a sequence of sub-sigma fields. Also, let $(X_n,Y_n)$ be a sequence of random variables on $(\Omega_n,\sF_n)$, and assume that for any $k\ge 1$ the conditional expectations  $U_{n,k}:=\mathbb{E}(X_n^k|\sG_n), V_{n,k}:=\mathbb{E}(Y_n^k|\sG_n)$ exist as finite random variables, such that
\begin{align}
\limsup_{n\rightarrow\infty}\P(|U_{n,k}-V_{n,k}|>\varepsilon)=0.
\label{eq:abc1}
\end{align}
Moreover, if there exists $\eta>0$ such that
\begin{align}
\limsup_{k\rightarrow\infty}\limsup_{n\rightarrow \infty}\P\left(\left|\frac {\eta^k}{k!}U_{n,k}\right|>\varepsilon\right)=0,
\label{eq:abc2}
\end{align}
then for any $t\in \R$, $\E e^{itX_n}-\E e^{itY_n}\rightarrow 0$.
\end{lem}

\begin{proof}
First  note that
$$\P\left(\left|\frac {\eta^k}{k!}V_{n,k}\right|>\varepsilon\right)\le \P\left(\left|\frac {\eta^k}{k!}U_{n,k}\right|>\frac{\varepsilon}{2}\right)+\P\left(|U_{n,k}-V_{n,k}|>\frac{\varepsilon k!}{2|\eta|^k}\right).
$$
Taking limits as $n\rightarrow \infty$, and using (\ref{eq:abc1}) and (\ref{eq:abc2}), it follows that
\begin{equation}
\limsup_{k\rightarrow\infty}\limsup_{n\rightarrow \infty}\P\left(\left|\frac {\eta^k}{k!}V_{n,k}\right|>\varepsilon\right)=0.
\label{ob3}
\end{equation}

To prove the lemma it suffices to show that for any $t\in \R$ and an non-negative integer $\ell$,
\begin{equation}\label{works?}
|\E (e^{itX_n}X_n^\ell|\sG_n)-\E (e^{itY_n}Y_n^\ell|\sG_n)|\stackrel{\sP}{\rightarrow}0.
\end{equation}
Indeed the lemma follows immediately from (\ref{works?}), as follows: Setting $\ell=0$,
$|\E (e^{itX_n}|\sG_n)-\E (e^{itY_n}|\sG_n)|$ converges to $0$ in probability. Since $|\E (e^{itX_n}|\sG_n)-\E (e^{itY_n}|\sG_n)|$ is also bounded by $2$ in absolute value, the dominated convergence theorem gives the desired result.

Therefore, it thus remains to prove the claim (\ref{works?}). For this, first assume $|t|\le \eta$, and let $\ell$ be fixed but arbitrary non-negative integer. By a Taylor's series expansion, for any $k\in\mathbb{N}$,
$$\left| e^{it}-\sum_{r=0}^{k-1} \frac{(it)^r}{r!}\right|\le \frac{\eta^k}{k!},$$
and so
\begin{align}
|\E (e^{itX_n}X_n^\ell|\sG_n)-\E (e^{itY_n}Y_n^\ell|\sG_n)|\le &\sum_{r=0}^{k-1}\frac{|t|^r}{r!}|U_{n,r+\ell}-V_{n,r+\ell}|+\frac{|\eta|^k}{k!}U_{n,k}+\frac{|\eta|^k}{k!}V_{n,k}.
\label{ob4}
\end{align}
From (\ref{ob4}) taking limits as $n\rightarrow\infty$ followed by $k\rightarrow\infty$, and using (\ref{eq:abc1}), (\ref{eq:abc2}) and (\ref{ob3}) gives $|\E (e^{itX_n}X_n^\ell|\sG_n)-\E (e^{itY_n}Y_n^\ell|\sG_n)|\stackrel{\sP}{\rightarrow}0$, for $|t|\le \eta$.

The proof of (\ref{works?}) is now completed by induction. Suppose the result holds for any $|t|\le p\eta$, for some $p\in \mathbb{N}$, and let $|t|\in (p\eta, (p+1)\eta]$. Then  setting $t_0:=t-\frac{t}{|t|}\eta$ we have that $|t-t_0|=\eta$, and $|t_0|\le p\eta$. Expanding in a Taylor's series around $t_0$ \begin{equation*}
|\E (e^{itX_n}X_n^\ell|\sG_n)-\E (e^{itY_n}Y_n^\ell|\sG_n)|\le \sum_{r=0}^{k-1}\frac{\eta^r}{r!}|\E (e^{it_0X_n}X_n^{r+\ell}|\sG_n)-\E (e^{it_0Y_n}Y_n^{r+\ell}|\sG_n)|+\frac{\eta^k}{k!}U_{n,k}+\frac{\eta^k}{k!}V_{n,k}.
\end{equation*}
Since $|t_0|\le p\eta$, letting $n\rightarrow\infty$ followed by $k\rightarrow\infty$ it follows by the induction hypothesis and (\ref{eq:abc1}) and (\ref{eq:abc2}) that $|\E (e^{itX_n}X_n^\ell|\sG_n)-\E (e^{itY_n}Y_n^\ell|\sG_n)|\stackrel{\sP}{\rightarrow}0.$
This completes the proof of (\ref{works?}) by induction.
\end{proof}

As mentioned earlier, the reason for Lemma \ref{abc} is that separate convergences of the conditional distribution $(Y_n|X_n)$ and the marginal distribution $X_n$ do not in general, imply that $Y_n$ converges in law to $Y$.
This is illustrated in the following example:

\begin{remark} Let $(X_n, Y_n)$ be a sequence of random variables such that $X_n \sim Uniform(0, 1/n)$, and given $X_n$, the conditional distribution of $Y_n$ is given by  
$$\P_n(Y_n\in A|X_n=x)=K(x,A)
:=\left\{
\begin{array}{cc}
\pmb  1\{0\in A\}  & \text{ if } x>0,   \\
\lambda(A) &    \text{ if } x=0;
\end{array}
\right.,
$$
for $A\in \sB([0,1])$,
where $\lambda(A)$ is the Lebesgue measure of $A$. 
In terms of random variables, this means that if $X_n=0$, then pick $Y_n$ to be $U(0,1)$, whereas if $X_n>0$, set $Y_n=0$. As a candidate for the joint limiting distribution, define random variables $(X, Y)$ such that $X=0$ almost surely. Define the conditional distribution of $Y$ given $X$ by the same kernel $K$ as above, that is, $\P(Y\in A|X=x)=K(x, A)$, for $A\in \sB([0,1])$ and $x \in [0, 1]$.
It is trivial to check that $X_n$ converges in distribution to $X$. Also since the conditional distribution of $(Y_n|X_n=x)$ is the same as the conditional distribution of $(Y|X=x)$, the convergence of conditional distribution is also immediate. However, $Y_n$ does not converge to $Y$ in distribution, as $Y_n=0$ almost surely, whereas $Y\sim U(0,1)$. 

Typically, conditions like the {\it weak-Feller} property of the kernel \cite{lasserre} or {\it set-wise} convergence of the marginals (via the Vitali-Hahn-Saks theorem \cite{royden}) are required for joint convergence. Nevertheless, in our case the proof of the unconditional convergence follows without having to invoke any such general theorems as sums of i.i.d. random variables can be dealt with directly, which in turn provide good approximations to the random variables studied in the paper.

\end{remark}

\end{document}